\documentclass[a4paper,11pt]{amsart}
\usepackage{amsfonts}
\usepackage{amssymb}
\usepackage[utf8]{inputenc}
\usepackage{amsmath}
\usepackage{pdflscape}
\usepackage{graphicx}
\usepackage[colorlinks=true, linkcolor=red, citecolor=blue]{hyperref}
\usepackage[]{epsfig}
\usepackage[]{pstricks}
\usepackage{tikz}

\makeatletter
\@namedef{subjclassname@2020}{\textup{2020} Mathematics Subject Classification}
\makeatother

\numberwithin{equation}{section}

\newtheorem{theorem}{Theorem}[section]

\newtheorem{lemma}[theorem]{Lemma}

\newtheorem{corollary}[theorem]{Corollary}

\newtheorem{proposition}[theorem]{Proposition}

\theoremstyle{definition}

\newtheorem{remark}[theorem]{Remark}

\theoremstyle{remark}

\newcommand{\R}{\mathbb{R}}

\def\be{\begin{equation}}
\def\ee{\end{equation}}
\def\ba{\begin{eqnarray}}
\def\ea{\end{eqnarray}}

\def\vf{\varphi}

\numberwithin{equation}{section}

\begin{document}

	\pagenumbering{arabic}	
\title[Stabilization and control of the NLS--KdV system]{Control results for a model of resonant interaction between short and long capillary-gravity waves}
\author[Capistrano-Filho]{Roberto de A. Capistrano-Filho*}
\address{Departamento de Matem\'atica,  Universidade Federal de Pernambuco (UFPE), 50740-545, Recife (PE), Brazil.}
\email{roberto.capistranofilho@ufpe.br}

\author[Pampu]{Ademir B. Pampu}
\address{Universidade Estadual da Paraíba - Centro de Ciências exatas e sociais aplicadas, 58706-550, Patos (PB), Brazil.}
\email{ademir@servidor.uebp.edu.br}
\thanks{*Corresponding author: roberto.capistranofilho@ufpe.br}
\subjclass[2020]{35Q55, 35Q53, 93B05, 93D15, 35A21} 
\keywords{Bourgain spaces · Global control results · Propagation of compactness · Propagation of regularity ·  Unique continuation property · Schrödinger--KdV system}

\begin{abstract}
The purpose of this article is the investigation of the global control properties of a coupled nonlinear dispersive system posed in the periodic domain $\mathbb{T}$, a system with the structure of a nonlinear Schr\"odinger equation and a nonlinear  Korteweg-de Vries equation. Combining estimates derived from Bourgain spaces and using microlocal analysis we show that this system has global control properties. The main novelty of this work is twofold. One is that the global results for the nonlinear system are presented for the first time thanks to the propagation of singularities. The second one is that these propagation results are shown to a coupled dispersive system with two equations defined by differential operators with principal symbols of different orders.
\end{abstract}
\maketitle

\section{Introduction}

\subsection{Setting of the problem}Our work is related to global control properties of a system composed by a nonlinear Schr\"odinger equation and a nonlinear Korteweg-de Vries equation
\begin{equation}\label{prob-intro}
\left\{\begin{array}{ll}
i \partial_{t} u+\partial_{x}^{2} u = i \partial_{x} v+\beta|u|^{2} u, & (x,t)\in\mathbb{T} \times \mathbb{R}_{+},\\
\partial_{t} v+\partial_{x}^{3} v + \frac{1}{2} \partial_{x}\left(v^{2}\right) +\mu\partial_{x}v = \operatorname{Re}\left(\partial_{x} u\right),  & (x,t)\in \mathbb{T} \times \mathbb{R}_{+},\\
u(x, 0)=u_{0}(x), \quad v(x, 0)=v_{0}(x), & x \in \mathbb{T},
\end{array}\right.
\end{equation} 
where $u = u(x,t)$ is a complex valued function and $v = v(x, t)$ is a real valued function.   
The nonlinear Schr\"odinger--Korteweg-de Vries system (NLS--KdV) appears in the study of resonant interaction between short and long capillary-gravity waves on water of uniform finite depth, in plasma physics and a diatomic lattice system. Here, $u$ represents the short wave, while $v$ stands for the long wave, see e.g. \cite{AlbertPava,Arbieto,Beney,CorchoLinares} and the references therein for more details about the physical motivation for this system.

The first goal of the manuscript is to establish well-posedness results for the system \eqref{prob-intro}. To do so the main ingredient is a fixed point argument in the Bourgain spaces associated with the linear Schr\"odinger and linear Korteweg-de Vries equations. Once we have the global well-posedness of the NLS--KdV system, we can consider the system \eqref{prob-intro} from a control point of view with two forcing terms $f$ and $g$ added in each equation
\begin{equation}\label{prob-intro-contr}
\left\{\begin{array}{ll}
i \partial_{t} u+\partial_{x}^{2} u = i \partial_{x} v+\beta|u|^{2} u + f, & (x,t)\in\mathbb{T} \times (0,T),\\
\partial_{t} v+\partial_{x}^{3} v + \frac{1}{2} \partial_{x}\left(v^{2}\right) + \mu\partial_{x}v = \operatorname{Re}\left(\partial_{x} u\right) + g,  & (x,t)\in \mathbb{T} \times (0,T),\\
u(x, 0)=u_{0}(x), \quad v(x, 0)=v_{0}(x), & x \in \mathbb{T},
\end{array}\right.
\end{equation}
where $f$ and $g$ are assumed to be supported in a non-empty subset $\omega \subset \mathbb{T}$. Here, the main point is to use microlocal analysis to prove the results of the propagation of singularities, which is the key point to proving global control results. The main difficulty to prove these propagations is related to the fact that we have a coupled system defined by two differential operators with principal symbols of different orders.

It is not of our acknowledgment any global control results for the nonlinear NLS--KdV system \eqref{prob-intro-contr}. However,  for the linear system control problems are considered in \cite{Arauna}. Precisely, the authors treated a linear Schrödinger--Korteweg-de Vries system 
$$
\left\{\begin{array}{lll}
i \partial_tw+\partial^2_xw=a_{1} w+a_{2} y+h \mathbf{1}_{\omega} & \text { in } & Q \\
\partial_ty+\partial^3_xy+\partial_x(M y)=\operatorname{Re}\left(a_{3} w\right)+a_{4} y +l\mathbf{1}_{\omega}  & \text { in } & Q \\
w(0, t)=w(1, t)=0 & \text { in } & (0, T), \\
y(0, t)=y(1, t)=\partial_xy(1, t)=0 & \text { in } & (0, T), \\
w(x, 0)=w_{0}(x), \quad y(x, 0)=y_{0}(x) & \text { in } & (0,1)
\end{array}\right.
$$
in a bounded domain $Q:=(0,1)\times(0,T)$ with a purely real or a purely imaginary control $h$ acting in the Schrödinger equation and a control $l$ acting in the KdV equation. Thanks to the Carleman estimates they proved an observability inequality that helps them achieve the result.  However, the following questions naturally arise:

\vspace{0.2cm}

\noindent\textbf{Control problems:} \textit{What can be said about the global controllability for system \eqref{prob-intro-contr}? Could it be possible to find appropriate damping mechanisms to stabilize this system ?}

\vspace{0.2cm}

In this way, our work will provide answers to these questions for the nonlinear NLS--KdV system \eqref{prob-intro-contr}. Although these issues are typical in control theory and have been the subject of study in several single equations over the past 30 years, the controllability problems involving nonlinear coupled dispersive systems still are not well understood. 

\subsection{Known results for single equations} Let us present a review of the control and stabilization results for the KdV and NLS equations. We caution that this is only a small sample of the extant works on these equations.

\subsubsection{KdV equation}Russell and Zhang are the pioneers in the study of control problems to the KdV equation \cite{Russel93,Russel96}. They treated the following KdV system 
\begin{equation}
u_{t}+uu_{x}+u_{xxx}=f\text{, }
\label{I5}%
\end{equation}
with periodic boundary conditions and an internal control $f$. Since then, both controllability and stabilization problems have been intensively studied. We can cite, for instance, the exact boundary controllability of KdV on a bounded domain \cite{cerpa,cerpa1,GG,Rosier,Zhang2} and the internal control problem  \cite{CaPaRo}, among others.

It is well known that the KdV system \eqref{I5} has at least the following conserved integral quantities
$$I_1(t)=\int_{\mathbb{T}}u(x,t)dx, \quad  I_2(t)=\int_{\mathbb{T}}u^2(x,t)dx \quad \text{and} \quad  I_3(t)=\int_{\mathbb{T}}\left(u_x^2(x,t)-\frac{1}{3}u^3(x,t)\right)dx.$$
From the historical origins of the KdV equation involving the behavior of water waves in a shallow channel, it is natural to think that $I_{1}$ and $I_{2}$  express conservation of \textit{volume (or mass)} and \textit{energy}, respectively.

The pioneering work in the periodic case is due to Russel and Zhang \cite{Russel93} and it is purely linear.  After some years and the discovery of a subtle smoothing property of solutions of the KdV equation due to Bourgain \cite{Bourgain1},  the authors can extend their results to the nonlinear system \cite{Russel96}. Precisely,  the authors studied  the equation \eqref{I5} assuming $f$ supported in a given open set $\omega
\subset\mathbb{T}$ and taking the control input $f(x,t)$ as follows
\begin{equation}
f\left(  x,t\right)  =\left[  Gh\right]  \left(  x,t\right)  :=g\left(
x\right)  \left(  h\left(  x,t\right)  -\int_{\mathbb{T}}g\left(  y\right)
h\left(  y,t\right)  dy\right),\label{I8}%
\end{equation}
where $h$ is considered as a new control input, and $g(x)$ is a given
non-negative smooth function such that $\{g>0\}=\omega$ and
\[
2\pi\left[  g\right]  =\int_{\mathbb{T}}g\left(  x\right)  dx=1\text{.}%
\]
For the chosen $g$, it is easy to see that
\[
\frac{d}{dt}\int_{\mathbb{T}}u\left(  x,t\right)  dx=\int_{\mathbb{T}}f\left(
x,t\right)  dx=0,
\]
for any $t\in\mathbb{R}$ and for any solution $u=u(x,t)$ of the system
\begin{equation}
u_{t}+uu_{x}+u_{xxx}=Gh. \label{I9}%
\end{equation}
Thus, the mass of the system is indeed conserved. With this in hand Russell and Zhang were able to show the local exact controllability and local exponential stabilizability for the system \eqref{I9}. Indeed, the results presented in \cite{Russel96} are essentially linear; they are more or less small perturbations of the linear results. However, Laurent \textit{et al.} in \cite{Laurent} showed global results for the system \eqref{I9}. The global control results are established with the aid of certain properties of propagation of compactness and regularity in Bourgain spaces.

\subsubsection{NLS equation} Consider the following equation
\begin{equation}\label{NLSCR}
i\partial_tu + \Delta u=\lambda |u|^2u,\quad (x,t)\in \mathcal{M}\times \mathbb{R}.
\end{equation}
The first results to the system \eqref{NLSCR}, when $\mathcal{M}$ a compact Riemannian manifold of dimension $2$ without boundary, is due to Dehman \textit{et al.} in \cite{dehman-gerard-lebeau}.  The authors considered the stabilization and exact controllability problem for NLS. Precisely, to prove the control properties, the authors were able to prove the propagation results in $\mathcal{M}$. However, these properties are shown considering $\omega$ be an open subset of $\mathcal{M}$ and the following two assumptions: 
\begin{itemize}
       \item[(A)] $\omega$ geometrically controls \(\mathcal{M}\); i.e. there exists \(T_{0}>0,\) such that every geodesic of \(\mathcal{M}\)
traveling with speed 1 and issued at \(t=0,\) enters the set \(\omega\) in a time \(t<T_{0}\) .
       
       \item[(B)] For every \(T>0,\) the only solution lying in the space \(C[0, T], H^{1}(\mathcal{M}) )\) of the system 
       \begin{equation*}
       \begin{cases}
i \partial_{t} u+\Delta u+b_{1}(x,t) u+b_{2}(x,t) \overline{u},& (x,t)\in\mathcal{M}\times(0,T),\\
u=0,& (x,t)\in \omega\times( 0, T) ,\end{cases}
\end{equation*}
where $b_{1}(t, x)$ and $b_{2}(t, x) \in L^{\infty}\left(0, T, L^{p}(\mathcal{M})\right)$ for some \(p>0\) large enough, is the trivial
one \(u \equiv 0\).
\end{itemize}

Considering the NLS on a periodic domain $\mathbb{T}$ with Dirichlet or Neumann boundary conditions, Laurent \cite{Laurent-esaim} applied the method introduced by Dehman \textit{et al.} to prove that this system is globally internally controllable.

When a compact Riemannian manifold of dimension $d \geq 3$ is considered, Strichartz estimates do not yield uniform well-posedness results at the energy level for the NLS equation, a property which seems to be very important to prove controllability results. In this way, Burq \textit{et al.} in two works \cite{Burq1,Burq2} managed to introduce the Bourgain spaces $X^{s, b}$ on certain manifolds without boundary where the bilinear Strichartz estimates can be shown and, consequently, they get the uniform well-posedness for the NLS equation. Taking advantage of these results, Laurent \cite{Laurent-siam} proved that the geometric control condition $(\mathcal{A})$  is sufficient to prove the exact controllability for the NLS in $X^{s, b}$ spaces on some three-dimensional Riemannian compact manifolds. Similarly in this work, we mention \cite{Rosier-Zhang} and \cite{Laurent-esaim} where controllability results were studied for the NLS in Euclidean and periodic domains, respectively, relying on the properties of the Bourgain spaces.

\subsection{Main results} Note that to deal with the nonlinearity associated with the KdV part in \eqref{prob-intro-contr}, we must use appropriated estimates which requires that 
the mean value of $v$ satisfying
$$
[v] = \frac{1}{2\pi}\int_{\mathbb{T}}{v(x, t)}dx=0, \quad \forall t\geq0.
$$
If we set $\mu := [v_{0}]$, we get from the second equation in \eqref{prob-intro-contr} that 
\begin{eqnarray*}
\tilde{\mu} := [v_{0}] = \frac{1}{2\pi}\int_{\mathbb{T}}{v(x, 0)}dx = \frac{1}{2\pi}\int_{\mathbb{T}}{v(x, t)}dx, \quad \forall t>0.
\end{eqnarray*}
Thus it is convenient to set $\widetilde{v} = v - \tilde{\mu}$ and to  study the following equivalent system 
\begin{equation*}
\left\{\begin{array}{ll}
i \partial_{t} u+\partial_{x}^{2} u =i \partial_{x}\widetilde{v}+\beta|u|^{2} u +f, & (x,t)\in\mathbb{T} \times (0,T),\\
\partial_{t}\widetilde{v}+\partial_{x}^{3}\widetilde{v}+ \frac{1}{2} \partial_{x}\left(\widetilde{v}^{2}\right)+ (\mu + \tilde{\mu}) \partial_{x} \widetilde{v}=\operatorname{Re}\left(\partial_{x} u\right)+g, & (x,t)\in\mathbb{T} \times (0,T),\\
u(x, 0)=u_{0}(x), \quad v(x, 0)=\widetilde{v}_{0}(x), & x \in \mathbb{T},
\end{array}\right.
\end{equation*}
where $\mu \in \mathbb{R}$ is a constant. 

Now consider  $a \in L^{\infty}(\mathbb{T})$ a real-valued function such that
\begin{eqnarray}\label{Def-a}
a(x)^{2} > \eta > 0,
\end{eqnarray}
in some non-empty open set $\omega \subset \mathbb{T}$ and the operator $G$ defined as in \cite{Russel93}, for some $g \in C^{\infty}(\mathbb{T})$ real-valued function such that $g > 0$ in $\omega \subset \mathbb{T}$, as 
\begin{eqnarray}\label{Def-G}
Gh(x, t) := g(x)\left(h(x,t) - \int_{\mathbb{T}}{g(y)h(y,t)}dy\right),
\end{eqnarray}
where $h$ is any function considered as a control input. In order to stabilize our system we have choosing  $f:=-i a(x)^{2} u$ and  $g:=Gh$, with  $h=-G^*\tilde{v}$,  so the following closed-loop system reads
\begin{equation}\label{prob-diss}
\left\{\begin{array}{ll}
i \partial_{t} u+\partial_{x}^{2} u =i \partial_{x} v+\beta|u|^{2} u-ia(x)^2u, & (x,t)\in\mathbb{T} \times \mathbb{R}_+,\\
\partial_{t} \tilde{v}+\partial_{x}^{3}\tilde{v}+\frac{1}{2} \partial_{x}\left(\tilde{v}^{2}\right) +  (\mu + \tilde{\mu}) \partial_{x}\tilde{v} = \operatorname{Re}\left(\partial_{x} u\right) -GG^*\tilde{v}, & (x,t)\in\mathbb{T} \times \mathbb{R}_+,\\
u(x, 0)=u_{0}(x), \quad  \tilde{v}(x, 0)= \tilde{v}_{0}(x), & x \in \mathbb{T}.
\end{array}\right.
\end{equation}
In this case, we can define the energy of the system as 
$$E(t):=\|u(t)\|_{L^{2}(\mathbb{T})}^{2} + \| \tilde{v}(t)\|_{L_{0}^{2}(\mathbb{T})}^{2},$$
and 
$L^{2}_{0}(\mathbb{T}) = \{w \in L^{2}(\mathbb{T}); [w] = 0\}$.  So, multiplying the first equation of  \eqref{prob-diss} by  $u$,  the second one by $\tilde{v}$ and integrating by parts we can obtain\footnote{Actually,  by \eqref{Def-G} and Fubini's theorem, $G$ is self-adjoint (i.e.  $G^*= G$),  so we shall keep the notation for the feedback $h=-G^*\tilde{v}= - G\tilde{v}$ throughout.}
\begin{equation}\label{derivative}
\frac{d}{dt}E(t) = - 2\left(\|a(\cdot)u(t)\|_{L^{2}(\mathbb{T})}^{2}  + \|G\tilde{v}(t)\|_{L^{2}_{0}(\mathbb{T})}^{2}\right) \leq 0.
\end{equation}
This indicates that the controls play the role of two damping mechanisms. So, our result establishes that the system \eqref{prob-diss} is asymptotically exponential stable. Precisely, we provide a positive answer to the global stabilization question already mentioned at the beginning of the introduction.

\begin{theorem}\label{est-teo_int}
Assume that $a \in L^{\infty}(\mathbb{T})$ satisfy the conditions \eqref{Def-a} and $Gh$ given by \eqref{Def-G}. Then, for every $R_{0} > 0$, there exist $C:=C(R_0)>0$ and $\gamma > 0$ such that the following inequality holds
\begin{eqnarray}
\|(u,  \tilde{v})(t)\|_{L^{2}(\mathbb{T}) \times L^{2}_{0}(\mathbb{T})} \leq Ce^{-\gamma t}\|(u_{0}, \tilde{v}_{0})\|_{L^{2}(\mathbb{T}) \times L^{2}_{0}(\mathbb{T})}, \quad \forall t \geq 0,
\end{eqnarray}
for every solution $(u, \tilde{v})$ of the system \eqref{prob-diss} with initial data $(u_{0},  \tilde{v}_{0}) \in L^{2}(\mathbb{T}) \times L^{2}_{0}(\mathbb{T})$ satisfying 
$
\|(u_{0}, \tilde{v}_{0})\|_{L^{2}(\mathbb{T}) \times L^{2}_{0}(\mathbb{T})} \leq R_{0}.
$
\end{theorem}

Once we have established the global stabilization result, the answer for the global exact controllability problem will be a consequence of the following local exact controllability theorem.

 \begin{theorem}\label{exc_local}
 Let $\omega$ be any nonempty set of $\mathbb{T}$. Then 
 there exist $\delta > 0$ and $T > 0$ such that for every 
  $(u_{0},\tilde{v}_{0}) \in L^{2}(\mathbb{T}) \times L^{2}_{0}(\mathbb{T})$ with 
$  \|(u_{0},  \tilde{v}_{0})\|_{L^{2}(\mathbb{T}) \times L^{2}_{0}(\mathbb{T})}
  < \delta$,
 we can find $f \in C([0, T]; L^{2}(\mathbb{T}))$ and $h\in C([0, T]; L^{2}(\mathbb{T}))$, $f$ and $h$ compactly supported  in $]0, T[ \times \omega$, such that, 
  the unique solution  $(u,  \tilde{v}) \in C([0, T]; L^{2}(\mathbb{T}) \times L^{2}(\mathbb{T}))$ 
  of the system \eqref{prob-diss} satisfies $(u, \tilde{v})(x,T) = (0, 0)$.
 \end{theorem}

Finally with the previous local exact controllability result in hand,  combining it with the global stabilization result we get the global controllability result, which can be read as follows.

\begin{theorem}\label{controle}
Let $\omega \subset \mathbb{T}$ be a nonempty open set and $R_{0} > 0$. Then, there exist $T:=T(R_{0})>0$ such that, for every 
$(u_{0}, v_{0}), (u_{1}, v_{1}) \in L^{2}(\mathbb{T}) \times L^{2}(\mathbb{T})$ with  
\begin{eqnarray*}
\|(u_{0}, v_{0})\|_{L^{2}(\mathbb{T}) \times L^{2}(\mathbb{T})} \leq R_{0} \quad \hbox{ and } \quad \|(u_{1}, v_{1})\|_{L^{2}(\mathbb{T}) \times L^{2}(\mathbb{T})} \leq R_{0},
\end{eqnarray*}
and $[v_{0}] = [v_{1}]$ one can find control inputs $f\in C([0, T]; L^{2}(\mathbb{T}))$ and $h \in C([0, T]; L^{2}_{0}(\mathbb{T}))$, with $f$ and $h$ supported in $\omega \times (0, T)$,  such that the unique solution $(u, v) \in C([0, T]; L^{2}(\mathbb{T}) \times L^{2}(\mathbb{T}))$ of the system  \eqref{prob-intro-contr} satisfies $(u,v)(x,T) = (u_{1}(x), v_{1}(x))$.
\end{theorem}

\subsection{Heuristic and outline of the manuscript} In this work global control results are proved by combining estimates derived from Bourgain spaces and using microlocal analysis. To our knowledge, this is the first time that this method is used in a coupled dispersive system with two equations defined by differential operators with principal symbols of different orders and more importantly, a key feature in this contribution is that we can prove global control results for the nonlinear system in a bounded domain. This represents an improvement concerning the previous paper \cite{Arauna} where the control problem, for the linear system, is considered. The key ingredients of this work are:
\begin{itemize}
\item[$\bullet$]\textit{Strichartz, bilinear and multilinear estimates} associated to the solution of the problem under consideration;
\item[$\bullet$]\textit{Microlocal analysis} to prove \textit{propagation of the regularity and compactness} for two equations defined by differential operators with principal symbols of different orders;
\item[$\bullet$]\textit{Unique continuation property} which is a consequence of the Carleman estimates for each equation, KdV and Schrödinger equations. 
\end{itemize}

The proof of the Theorem \ref{est-teo_int} is equivalent to prove an \textit{observability inequality}, which one, by using \textit{contradiction arguments}, relies on to prove a \textit{unique continuation property} for the system \eqref{prob-diss}. This property is achieved thanks to the \textit{propagation results} using the smooth properties of the Bourgain spaces. The main difficulty to prove the propagation results arises from the fact that the system \eqref{prob-diss} is defined by two differential operators with principal symbols of different orders. To overcome this difficulty we employ the estimates, proved in Section \ref{Sec2}, for the solution of our problem in Bourgain spaces. 

The strategy to prove Theorem \ref{exc_local}  is to consider the control operator for the nonlinear problem as a perturbation of the control operator for the linear system associated with \eqref{prob-diss}, the perturbation argument is due to Zuazua \cite{Zuazua}. Additionally, the control result for large data (Theorem \ref{controle}) will be a combination of a global stabilization result (Theorem \ref{est-teo_int}) and the local control result (Theorem \ref{exc_local}), as it is usual in control theory.

It is important to point out that the dissipation laws in \eqref{prob-diss} are intrinsically linked with the physical problems modeled by the Schrödinger and the KdV equations. In the case when it is considered only one equation, these same dissipation laws were already considered in \cite{Laurent-esaim,Rosier-Zhang}, for the NLS equation, and in \cite{Laurent,Russel96}, for the KdV equation.  

Lastly, since we are working with a coupled system with the structure of the nonlinear Schrödinger equation and nonlinear Korteweg-de Vries equation it is natural to adapt the approach introduced by \cite{Laurent-esaim, Laurent}, however, a direct application of these results is not enough to get the main results of our article.  The first novelty of this manuscript is to deal with the coupled terms, here it is necessary estimates these terms in the Bourgain spaces (see Lemma \ref{LemaBou} and Proposition \ref{Prop_Bou}) to be able to perform with the fixed point argument, which is necessary to prove the well-posedness and the controllability results for the nonlinear system.  A second and important point is to treat the nonlinear case, instead of the linear one (see \cite{Arauna}), to deal with the nonlinear problem it is necessary to prove propagation of compactness and regularity results, and also, a unique continuation property for the solution of the system \eqref{prob-intro-contr}. This was achieved thanks to the estimates in the Bourgain spaces,  which are paramount and the key to dealing with the nonlinear terms.

Our manuscript is outlined as follows. Section \ref{Sec2} is to establish estimates needed in our analysis, \textit{Strichartz, bilinear and multilinear estimates}. With these estimates, we can prove the existence of solutions for the nonlinear NLS--KdV system with source and damping terms in Section \ref{Sec3}. Next, Sections \ref{Sec4} and \ref{Sec5}, are aimed to present the proof of the stabilization and controllability theorems, respectively. We also present in  Section \ref{Sec6} concluding remarks and open problems. Finally,  Appendixes \ref{Appendix1} and \ref{Appendix2} are devoted to proving the \textit{propagation results} and the \textit{unique continuation property}, respectively.

\section{Fourier transform restriction} \label{Sec2}
In this section we present definitions, some properties of the Bourgain spaces associated with the system \eqref{prob-intro}, and estimates which will be essential to prove our main results. It is well known that J. Bourgain \cite{Bourgain} discovered a subtle smoothing property of solutions of the Schrödinger and KdV equations in $\mathbb{R}$ and $\mathbb{T}$.

 \subsection{Definitions and notations}
Let us now define the Bourgain spaces $X^{k, b}$ and $Y^{s, b}$ associated with the linear operators of the Schr\"odinger and KdV equations, respectively.  For given $k,b,s\in\mathbb{R}$, functions $u:\mathbb{T\times
R\to C}$ and $v:\mathbb{T\times
R\to R}$, in $\mathcal{S}(\mathbb T\times\R)$, and $\mu \in \mathbb{R}$ define the quantities
$$
\|u\|_{X^{k, b}} = \left(\sum_{n \in \mathbb{Z}} \int_{-\infty}^{+\infty}\langle n\rangle^{2 k}\left\langle\tau+n^{2}\right\rangle^{2 b}|\hat{f}(n, \tau)|^{2} d \tau\right)^{1 / 2}$$
and
$$
\|v\|_{Y^{s, b}}=\left(\sum_{n \in \mathbb{Z}} \int_{-\infty}^{+\infty}\langle n\rangle^{2 s}\left\langle\tau-n^{3} + \mu n\right\rangle^{2 b}|\hat{g}(n, \tau)|^{2} d \tau\right)^{1 / 2},
$$
where $\hat{f}$ is the Fourier transform of $f$ with respect to the 
variables $x$ and $t$, with $\left\langle \cdot\right\rangle =\sqrt{1+\left\vert \text{ }\cdot\text{ }\right\vert ^{2}}$.  
The Bourgain spaces $X^{k,b}$ and $Y^{s,b}$  associated with Schrödinger and KdV operators are the completion of the Schwartz space $\mathcal{S}\left(  \mathbb{T\times R}\right)  $ under the norm $\left\Vert u\right\Vert _{X^{k,b}}$ and $\left\Vert
v\right\Vert _{Y^{s,b}}$, respectively.

Considering the time interval $I \subset \mathbb{R}$, we 
define $X^{k,b}(I)$ and $Y^{s,b}(I)$ as the restriction spaces of $X^{k,b}$ and $Y^{s,b}$, respectively, to the interval $I$ with the norm
\begin{equation}\label{rest-spaces}
\|f\|_{X^{k, b}(I)}=\inf _{\left.\tilde{f}\right|_{I}=f}\|\tilde{f}\|_{X^{k, b}} \quad \text { and } \quad \|g\|_{Y^{s, b}(I)}=\inf _{\left.\tilde{g}\right|_{I}=g}\|\tilde{g}\|_{Y^{s, b}}.
\end{equation}
We shall also consider the smaller spaces $\widetilde{X}^{k}$ and $\widetilde{Y}^{k}$ defined by the norms 
\begin{equation}\label{X_k}
\|u\|_{\tilde{X}^{k}}:=\|u\|_{X^{k, 1 / 2}}+\left\|\langle n\rangle^{k} \hat{u}(n, \tau)\right\|_{l_{n}^{2} L_{\tau}^{1}} \text { and }\|v\|_{\tilde{Y}^{s}}:=\|v\|_{Y^{s, 1 / 2}}+\left\|\langle n\rangle^{s} \hat{v}(n, \tau)\right\|_{l_{n}^{2} L_{\tau}^{1}}
\end{equation}
and the restriction spaces $\widetilde{X}^{k}(I)$ and 
$\widetilde{Y}^{k}(I)$ as in \eqref{rest-spaces}.
The companion spaces will be defined as  $Z^{k}$ and $W^{s}$ via the norms 
$$
\|u\|_{Z^{k}} = \|u\|_{X^{k,-1 / 2}}+\left\|\frac{\langle n\rangle^{k} \hat{u}(n, \tau)}{\left\langle\tau+n^{2}\right\rangle}\right\|_{ L_{\tau}^{1}l_{n}^{2}}
$$
and
$$\|v\|_{W^{s}} = \|v\|_{Y^{s,-1 / 2}} + \left\|\frac{\langle n\rangle^{s} \hat{v}(n, \tau)}{\left\langle\tau-n^{3} + \mu n\right\rangle}\right\|_{L_{\tau}^{1}l_{n}^{2}},$$
respectively. 

\subsubsection{Notations} If $I = (0, T)$, we denote $X^{k, b}(I)$ (resp. $Y^{s,b}(I)$, $\widetilde{X}^{k,b}(I)$ and $\widetilde{Y}^{s,b}(I)$) by $X^{k,b}_{T}$ (resp. $Y^{s,b}_{T}$, $\widetilde{X}^{k,b}_{T}$ and $\widetilde{Y}^{s,b}_{T}$). Denote from now on $\psi \in C_{0}^{\infty}(\mathbb{R})$ a  non-negative function supported in $[-2, 2]$  with $\psi \equiv 1$ on $[-1, 1]$ and  $\psi_{\delta}(t) := \psi\left({t}/{\delta}\right)$, for any $\delta > 0$. We also denote $$H_{0}^{s}(\mathbb{T})=\left\{u \in H^{s}(\mathbb{T}) ;[u]=0\right\} \quad \text{ and } \quad L_{0}^{2}(\mathbb{T})=\left\{u \in L^{2}(\mathbb{T}) ;[u]=0\right\}.$$
 Finally,  for any $a \in \mathbb{R}$, $a+$ (resp. $a-$) is a number slightly larger (resp. smaller) than $a$, more precisely, $a+ = a + \epsilon$, for some $\epsilon > 0$ as small as we want. We shall also  denote by $D^{r}$ the operator defined on $\mathcal{D}^{\prime}(\mathbb{T})$ by
\begin{equation}\label{DR}
\widehat{D^{r} u}(k)=\left\{\begin{array}{ll}
|k|^{r} \hat{u}(k) & \text { if } k \neq 0, \\
\hat{u}(0) & \text { if } k=0.
\end{array}\right.
\end{equation}
 
 \vspace{0.2cm}
 
 The following properties of $X^{k,b}_{T}$ and  $Y^{s,b}_{T}$, $k,s,b \in \mathbb{R}$, can be verified.
 \begin{itemize}
     \item[$(i)$] $X^{k, b}_{T}$ and $Y^{s,b}_{T}$ are Hilbert spaces.
     \vspace{0.1cm}
     \item[$(ii)$] Let $\delta, \delta_{1} \geq 0$, then $X^{k + \delta, b + \delta_{1}}_{T}$ is 
  continuously embedded in $X^{k, b}_{T}$ and if $\delta > 0$ and $\delta_{1} > 0$ we have a compact embedding. Similar result holds for $Y^{s,b}_{T}$ spaces.
       \vspace{0.1cm}
  \item[$(iii)$] The spaces $\widetilde{X}^{k}_{T}$ and $\widetilde{Y}^{s}_{T}$ are continuously embedded in $C([0, T]; H^{k}(\mathbb{T}))$ and $C([0, T]; H^{s}(\mathbb{T}))$, respectively. 
 \end{itemize}
 
\subsection{Overview of the estimates} This part of the work is dedicated to present estimates for the linear group associated with Schrödinger and KdV equations. Additionally,  we will present the trilinear and bilinear estimates associated with the nonlinearities involved in our problem, which are the key to proving the global control results in this manuscript. These results will be borrowed from \cite{Arbieto,Laurent,Laurent-esaim}.

Let $U(t) =e^{it\partial_{x}^{2}}$ and $V(t)=e^{-t(\partial_{x}^{3} + \mu \partial_{x})}$ be the unitary groups
associated with the linear Schrödinger and the Airy equations, respectively.  So, the following linear estimates are verified and the proofs were made in \cite[Lemma 4.1.]{Arbieto}, \cite[Lemma 3.4.]{Laurent} and \cite[Lemma 1.2.]{Laurent-esaim}, respectively, so we will omit them. 
 \begin{proposition} \label{Est-Lin} 
  Let $u_{0} \in H^{k}(\mathbb{T})$ and $v_{0} \in H^{s}(\mathbb{T})$, then 
  \begin{itemize}
       \item[$(i)$] $\|\psi_{1}(t)e^{it\partial_{x}^{2}}u_{0}\|_{X^{k,b}} \leq C\|u_{0}\|_{H^{k}}$\quad and $$\left\|\psi_{T}(t)\int_{0}^{t}{e^{i(t-t')\partial_{x}^{2}}F(t')dt'}\right\|_{X^{k, \frac{1}{2}}} \leq C \|F\|_{Z^{k}}.$$ 
       
       \item[$(ii)$] $\|\psi_{1}(t)e^{ -t(\partial_{x}^{3} + \mu \partial_{x})}v_{0}\|_{Y^{s,b}} \leq C\|v_{0}\|_{H^{s}}$ \quad and $$\left\|\psi_{T}(t)\int_{0}^{t}{e^{(t-t')(\partial_{x}^{3} + \mu \partial_{x})}G(t')dt'}\right\|_{Y^{s, \frac{1}{2}}} \leq C \|G\|_{W^{s}}.$$
       
       \item[$(iii)$] Let $-1 \leq b \leq 1$, $s \in \mathbb{R}$ and $\varphi \in C^{\infty}(\mathbb{T})$. Then, for every $u \in Y^{s,b}$, $\varphi(x)u \in Y^{s -|b|, b}$. Similarly, the multiplication $\varphi$ maps $Y^{s,b}_{T}$ into $Y^{s - |b|, b}_{T}$.
       
              \item[$(iv)$] Let $-1 \leq b \leq 1$, $k \in \mathbb{R}$ and $\varphi \in C^{\infty}(\mathbb{T})$. Then, for every $u \in X^{k,b}$, $\varphi(x)u \in X^{k -|b|, b}$. Similarly, the multiplication $\varphi$ maps $X^{k,b}_{T}$ into $X^{k - |b|, b}_{T}$.
  \end{itemize}
 \end{proposition}
 
 The next proposition showed in \cite[Lemma 4.2.]{Arbieto} and \cite[Lemma 3.3.]{Laurent}, concerns the trilinear estimates for the term $|u|^{2} u$ and bilinear estimates for the term $\partial_xv^2$, respectively. 
 
 \begin{proposition}\label{Est-Mul}
The following estimates hold.
 \begin{itemize}
     \item[$(i)$] For any $k \geq 0$, 
     \begin{equation}\label{Est-Tri}
\|\psi(t) u v \bar{w}\|_{Z^{k}} \lesssim\|u\|_{X^{k, 3 / 8}}\|v\|_{X^{k, 3 / 8}}\|w\|_{X^{k, 3 / 8}}.
\end{equation}

 \item[$(ii)$] For any $s \geq 0$, $T \in (0,1)$ and $v_{1}, v_{2} \in Y^{s, \frac{1}{2}}$ x-periodic functions having zero x-mean for all t, that is $[v_{1}] = [v_{2}] = 0$,  there exist two constants $\theta > 0$ and $C > 0$, such that
 \begin{equation}\label{Est-Bi}
\left\| \psi_{T}(t) \partial_{x}\left(v_{1} v_{2}\right)\right\|_{W^{s}} \leq CT^{\theta}\left\|v_{1}\right\|_{Y^{s, 1 / 2}}\left\|v_{2}\right\|_{Y^{s, 1 / 2}},
\end{equation}
where $C$ is independent of $T$, $v_{1}$ and $v_{2}$.
 \end{itemize}
 \end{proposition}

At this point we present an elementary lemma concerning the stability of Bourgain’s spaces with respect to time localization which was proved in \cite[Lemma 4.4.]{Arbieto}.

\begin{proposition}\label{lema-bour}
Let $X_{\tau=h(\xi)}^{s, b}:=\left\{f:\langle\tau-h(\xi)\rangle^{b}\langle\xi\rangle^{s}|\hat{f}(\tau, \xi)| \in L^{2}\right\}$. Then, 
$$
\| \psi(t) f \|_{X^{s,b}_{\tau = h(\xi)}} \lesssim_{\psi, b} \|f\|_{X^{s,b}_{\tau = h (\xi)}}
$$
for any $x,b \in \mathbb{R}$. Furthermore, if $- \frac{1}{2} < b' \leq b < \frac{1}{2}$, then for any $0 < T < 1$ we 
have 
$$
\| \psi_{T}(t) f \|_{X^{s,b'}_{\tau = h(\xi)}} \lesssim_{\psi, b', b} T^{b-b'}\| f\|_{X^{s,-b}_{\tau = h(\xi)}}
$$
\end{proposition}

 Finally, to end this subsection, we invoke the following result related to the linear system associated with the KdV equation. Precisely, we present below the following Strichartz estimate for the KdV equation shown in \cite[Lemma 3.2.]{Laurent}.

\begin{proposition}\label{St-KdV}
Let $T > 0$. The following estimate is verified
\begin{eqnarray*}
\|v\|_{L^{4}(0, T; L^{4}(\mathbb{T}))} \leq C\|v\|_{Y_{T}^{0, \frac{1}{3}}}.
\end{eqnarray*}
In other words, $Y_{T}^{0, \frac{1}{3}}$ is continuously embedded in $L^{4}(0, T; L^{4}(\mathbb{T}))$. 
\end{proposition}

\subsection{Auxiliary estimates} The next auxiliary lemma helps us to prove the necessary estimates for the following derivative terms $\partial_xu$ and $\partial_xv$ in $ W^k_T$ and $Z^s_T$, $k, s \in \mathbb{R}$, respectively. 
 
 \begin{lemma}\label{LemaBou}
 Let $\epsilon > 0$ and $H: \mathbb{Z} \times \mathbb{R} \to \mathbb{R}$ such that
 \begin{eqnarray*}
 H(n, \tau) = \frac{|n|}{\left\langle \tau - n^{3} + \mu n\right\rangle^{\frac{1}{2} - \epsilon} \left\langle \tau + n^{2} \right\rangle^{\frac{1}{2} - \epsilon}}, \quad (n, \tau) \in \mathbb{Z} \times \mathbb{R},
 \end{eqnarray*}
 then there exists a constant  $C > 0$ such that 
 \begin{eqnarray*}
 H(n, \tau) \leq C, \quad \forall (n, \tau) \in \mathbb{Z} \times \mathbb{R}.
 \end{eqnarray*}
 \end{lemma}
 \begin{proof}
 It is enough to prove that, for every $(n, \tau) \in \mathbb{Z} \times \mathbb{R}$, with $|n|$ large enough, there exist $\alpha > 0$ and $C_{0} > 0$ such that 
 \begin{eqnarray}\label{Bou1}
 \langle \tau - n^{3} + \mu n \rangle \langle \tau + n \rangle \geq C_{0} |n|^{\alpha},
 \end{eqnarray}
 with $\left(\frac{1}{2} - \epsilon \right)\alpha > 1$.

 Indeed, let us assume that \eqref{Bou1} is verified. We then split the analysis into two cases, $|n| \leq C$, for some $C > 0$ large enough, and $|n| > C$. For the first case, note that considering a fixed $n_{0} \in \mathbb{Z}$, we get
 \begin{eqnarray}\label{Bou2}
 H(n_{0}, \tau) = \frac{|n_{0}|}{\left\langle \tau - n^{3}_{0} + \mu n_{0}\right\rangle^{\frac{1}{2} - \epsilon} \left\langle \tau + n^{2}_{0} \right\rangle^{\frac{1}{2} - \epsilon}} \leq |n_{0}|,
 \end{eqnarray}
  for all $\tau \in \mathbb{R}$. 
%
On the other hand, \eqref{Bou1} asserts that  there exist $C_{0} > 0$ and $\alpha>0$ such that
 \begin{eqnarray*}
 H(n, \tau) \leq \frac{|n|}{C_{0}|n|^{\left(\frac{1}{2} - \epsilon\right)\alpha}} \leq \frac{1}{C_{0}},
 \end{eqnarray*}
for every $(n, \tau) \in \mathbb{Z} \times \mathbb{R}$, with $|n| > C$. So, by \eqref{Bou1} and \eqref{Bou2}, 
\begin{eqnarray*}
H(n, \tau) \leq \max\left\{\frac{1}{C_{0}}, C\right\}, \quad \forall (n, \tau) \in \mathbb{Z} \times \mathbb{R},
\end{eqnarray*}
and the lemma is proved. 

Now, to prove \eqref{Bou1} we start by observing that
\begin{eqnarray*}
\langle \tau - n^{3} + \mu n \rangle \langle \tau + n^{2} \rangle &=& \left(1 + |\tau - n^{3} + \mu n|^{2} \right)^{\frac{1}{2}} \left(1 + |\tau + n^{2}| \right)^{\frac{1}{2}}\\ 
&\geq& \frac{1}{4}(1 + |\tau - n^{3} + \mu n|) (1 + |\tau + n^{2}|)  \\
& \geq & \frac{1}{4} \left( |\tau - n^{3} + \mu n| + |\tau + n^{2}|\right) \\
& \geq & \frac{1}{4}|n^{3} - n^{2} - \mu n| \\
& \geq & \frac{1}{4}\left(|n|^{3} - |n^{2} + \mu n|\right) \\
& = & \frac{1}{2}|n|^{3} + \frac{1}{2}|n|^{3} - |n^{2} + \mu n| \\
& \geq & \frac{1}{8}|n|^{3} + \frac{1}{8}|n|^{3} - \frac{1}{4}|n|^{2} - \frac{1}{4}|\mu||n|.
\end{eqnarray*} 
Since we have that $\displaystyle\lim_{|n| \to \infty}\left(\frac{1}{8}|n|^{3} - \frac{1}{4}|n|^{2} - \frac{1}{4}|\mu||n|\right) = \infty$, \eqref{Bou1} is verified. 
 \end{proof}

 We are in a position to prove the last linear estimates.
 \begin{proposition}\label{Prop_Bou}
 Let $T > 0$. If $u \in X^{k, \frac{1}{2}}_{T}$ and $v \in Y^{s, \frac{1}{2}}_{T}$, for $k, s \in \mathbb{R}$, we have
 \begin{eqnarray} \label{Bou-misto1}
 \|\partial_{x}u\|_{W^{k}_{T}} &\lesssim& \|u\|_{X^{k, \frac{1}{2}-}_{T}}
  \end{eqnarray}
  and
   \begin{eqnarray} 
 \label{Bou-misto2}
 \|\partial_{x}v\|_{Z^{s}_{T}} &\lesssim& \|v\|_{Y^{s, \frac{1}{2}-}_{T}}  .
 \end{eqnarray}
 \end{proposition}
 
 \begin{proof}
 We start observing that
 \begin{eqnarray*}
\|\partial_{x}u\|_{W^{k}} = \left\|\langle \tau - n^{3} + \mu n\rangle^{-\frac{1}{2}} \langle n\rangle^{k}\widehat{\partial_{x}u}\right\|_{L^{2}_{\tau}l^{2}_{n}} + \left\|\frac{\langle n\rangle^{k}\widehat{\partial_{x}u}(n, \tau)}{\left\langle\tau-n^{3} + \mu n\right\rangle}\right\|_{L^{1}_{\tau}l^{2}_{n}}.
\end{eqnarray*}
Since $\widehat{\partial_{x}u}(n, \tau) = in\widehat{u}(n, \tau)$ and
\begin{eqnarray*}
\frac{1}{\langle \tau - n^{3} + \mu n \rangle^{\frac{1}{2}}} \leq \frac{1}{\langle \tau - n^{3} + \mu n \rangle^{\frac{1}{2} - \epsilon}},
\end{eqnarray*}
for any $\epsilon > 0$, we have that
\begin{eqnarray*}
\left\|\langle \tau - n^{3} + \mu n\rangle^{-\frac{1}{2}} \langle n\rangle^{k}\widehat{\partial_{x}u}\right\|_{L^{2}_{\tau}l^{2}_{n}} \leq \left\|\frac{n}{\langle \tau - n^{3} + \mu n \rangle^{\frac{1}{2} - \epsilon}}\langle n\rangle^{k} \widehat{u} \right\|_{L^{2}l^{2}}.
\end{eqnarray*}
Thanks to the Lemma \ref{LemaBou} and using the duality argument, follows that
\begin{equation}\label{Bou3}  
\begin{split}
\left\|\frac{n}{\langle \tau - n^{3} + \mu n \rangle^{\frac{1}{2} - \epsilon}}\langle n\rangle^{k}\widehat{u} \right\|_{L^{2}l^{2}} =& \sup_{\|\varphi\|_{L^{2}l^{2}} \leq 1}\sum_{n \in \mathbb{Z}}\int_{-\infty}^{\infty}{\frac{|n|}{\langle \tau - n^{3} + \mu n \rangle^{\frac{1}{2} - \epsilon}}\langle n\rangle^{k}\widehat{u}(n, \tau)\overline{\widehat{\varphi}(n, \tau)}}d\tau \\ 
 = & \sup_{\|\varphi\|_{L^{2}l^{2}} \leq 1}\sum_{n \in \mathbb{Z}}\int_{-\infty}^{\infty}{H(n,\tau) \langle \tau + n^{2}\rangle^{\frac{1}{2} - \epsilon} \langle n\rangle^{k}\widehat{u}(n, \tau)\overline{\widehat{\varphi}(n, \tau)}}d\tau \\ 
\leq &\ C\sup_{\|\varphi\|_{L^{2}l^{2}} \leq 1}\sum_{n \in \mathbb{Z}}\int_{-\infty}^{\infty}{ \langle \tau + n^{2}\rangle^{\frac{1}{2} - \epsilon} \langle n\rangle^{k}\widehat{u}(n, \tau)\overline{\widehat{\varphi}(n, \tau)}}d\tau \\ 
 \leq &\ C\sup_{\|\varphi\|_{L^{2}l^{2}} \leq 1}\|\langle \tau + n^{2}  \rangle^{\frac{1}{2} - \epsilon}\langle n\rangle^{k}\widehat{u}\|_{L^{2}_{\tau}l^{2}_{n}}\|\varphi\|_{L^{2}_{\tau}l^{2}_{n}} \\ 
 \leq &\ C\|u\|_{X^{k, \frac{1}{2} -\epsilon}}.
\end{split}
\end{equation}
Finally, using the Cauchy-Schwarz inequality we get
\begin{equation*}
\begin{split}
\left\|\frac{\langle n\rangle^{k}\widehat{\partial_{x}u}(n, \tau)}{\left\langle\tau-n^{3} + \mu n\right\rangle}\right\|_{L^{1}_{\tau}l^{2}_{n}}
 \leq & \left\|\frac{1}{\left\langle\tau-n^{3} + \mu n\right\rangle^{\frac{1}{2} + \epsilon} \langle\tau-n^{3} + \mu n\rangle^{\frac{1}{2} - \epsilon}}\langle n\rangle^{k}\widehat{\partial_{x}u}(n, \tau)\right\|_{L^{1}_{\tau}l^{2}_{n}} \\ 
 \leq & \left\|\frac{1}{\langle\tau-n^{3} + \mu n\rangle^{\frac{1}{2} - \epsilon}}\langle n\rangle^{k}\widehat{\partial_{x}u}(n, \tau)\right\|_{L^{2}_{\tau}l^{2}_{n}}\\ \leq& \|u\|_{X^{k, \frac{1}{2} - \epsilon}},
\end{split}
\end{equation*}
where the last inequality holds using \eqref{Bou3}. So, \eqref{Bou-misto1} is verified. 

Analogously, we can prove \eqref{Bou-misto2}. In fact, observing that 
$$
\|\partial_{x}v\|_{Z^{s}} = \|\partial_{x}v\|_{X^{s, - \frac{1}{2}}} + \left\|\frac{\langle n\rangle^{s} \widehat{\partial_{x}v}(n, \tau)}{\left\langle\tau + n^{2}\right\rangle}\right\|_{L_{\tau}^{1}l_{n}^{2}}
$$
and
\begin{equation*}
\begin{split}
\left\|\frac{\langle n\rangle^{s} \widehat{\partial_{x}v}(n, \tau)}{\left\langle\tau + n^{2}\right\rangle}\right\|_{L_{\tau}^{1}l_{n}^{2}} =& \left\|\frac{ |n| }{\left\langle\tau + n^{2}\right\rangle^{\frac{1}{2} + \epsilon} \left\langle\tau + n^{2}\right\rangle^{\frac{1}{2} - \epsilon}}\langle n\rangle^{s} \widehat{v}(n, \tau) \right\|_{L_{\tau}^{1}l_{n}^{2}} \\ 
 \leq & \left\|\frac{ |n| }{ \left\langle\tau + n^{2}\right\rangle^{\frac{1}{2} - \epsilon}} \langle n\rangle^{s}\widehat{v}(n, \tau) \right\|_{L_{\tau}^{2}l_{n}^{2}} \\
 = & \left\|\frac{ |n| }{ \left\langle\tau + n^{2}\right\rangle^{\frac{1}{2} - \epsilon}\left\langle\tau - n^{3} + \mu n\right\rangle^{\frac{1}{2} - \epsilon}} \left\langle\tau - n^{3} + \mu n\right\rangle^{\frac{1}{2} - \epsilon}\langle n\rangle^{s}\widehat{v}(n, \tau) \right\|_{L_{\tau}^{2}l_{n}^{2}}, 
\end{split}
\end{equation*}
we conclude the proof of \eqref{Bou-misto2} arguing by duality as in \eqref{Bou3} and so the lemma is proved.
 \end{proof}

 \section{Existence of solutions for the NLS--KdV system}\label{Sec3}
In this section, we are able to prove that the following system 
 \begin{equation}\label{NLS-KDV-Ex}
\left\{\begin{array}{ll}
i \partial_{t} u+\partial_{x}^{2} u + i\varphi(t)^{2}a(x)^{2}u = i \partial_{x}v+\beta|u|^{2} u, & (x,t)\in\mathbb{T} \times \mathbb{R}_{+},\\
\partial_{t} v+\partial_{x}^{3} v+\frac{1}{2} \partial_{x}\left(v^{2}\right) + \mu\partial_{x}v = Re(\partial_{x}u)  - GG^{*}v, & (x,t)\in\mathbb{T} \times \mathbb{R}_{+},\\
u(x, 0)=u_{0}(x), \quad v(x, 0)=v_{0}(x), & x\in\mathbb{T},
\end{array}\right.
\end{equation}
where $a(x)$ satisfies \eqref{Def-a} and $Gv$ is defined in \eqref{Def-G}, is well-posed in $H^{s}$, for $s\geq0$. Precisely, we prove the result for $\widetilde{X}^{s}_T$ and $\widetilde{Y}^{s}_T$, defined by \eqref{X_k}. So, the result of well-posedness can be read as follows.
 \begin{theorem}\label{Teo-Ex}
 Let $T > 0$, $s \in \mathbb{R}_{+}$, $\beta, \mu \in \mathbb{R}$, $a \in C^{\infty}(\mathbb{T})$ and $\varphi\in C^{\infty}_{0}(\mathbb{R})$ taking real values. For a given $(u_{0}, v_{0}) \in H^{s}(\mathbb{T}) \times H^{s}_{0}(\mathbb{T})$ 
 there exists a unique solution $(u,v) \in \widetilde{X}^{s}_{T}\times \left(\widetilde{Y}^{s}_{T} \cap L^{2}(0, T; L^{2}_{0}(\mathbb{T})) \right)$ of \eqref{NLS-KDV-Ex} 
 and the same result is valid for $s = 0$ if $a \in L^{\infty}(\mathbb{T})$.
 Additionally, the flow map 
\begin{eqnarray} \label{flow}
F: L^{2}(\mathbb{T})\times L^{2}_{0}(\mathbb{T}) 
&\to& \widetilde{X}^{0}_{T} \times \widetilde{Y}^{0}_{T} \\ \nonumber
 (u_{0}, v_{0}) &\mapsto& F (u_{0}, v_{0}) = (u,v)
\end{eqnarray}
 is Lipschitz on every bounded set.  
 \end{theorem}
 \begin{proof} We split the proof into three parts. Precisely, first, we prove the existence and uniqueness of a local solution. After that, we prove that the solution of \eqref{NLS-KDV-Ex} is globally defined. Finally, we can derive that the flow map \eqref{flow} is Lipschitz.  The idea of the proof is to follow the approach of  \cite{Laurent-esaim,Laurent} which was inspired by Bourgain  \cite{Bourgain,Bourgain1}. A main difference here is the necessity of Proposition \ref{Prop_Bou} which is essential for the proof of the fixed point theorem.  For the sake of completeness let us present the details of the proof.
 
\begin{itemize}
\item[(a)]\textbf{Existence and uniqueness of local solution.}
\end{itemize}
  
 Let $T > 0$ to be determined later and $t \in [0, T]$, system \eqref{NLS-KDV-Ex} is equivalent to the following Duhamel integral equations
$$
 u(t) =e^{it\partial_{x}^{2}}u_{0} - i\int_{0}^{t}{e^{i(t-t')\partial_{x}^{2}}\left\{i\partial_{x}v(t') + \beta|u|^{2}u(t') - i\varphi(t')^{2}a(x)^{2}u(t')  \right\}}dt' 
 $$
 and
 $$ 
 v(t) =e^{-t(\partial_{x}^{3} + \mu\partial_{x})}v_{0} + \int_{0}^{t}{e^{-(t-t')(\partial_{x}^{3} + \mu\partial_{x})}\left\{Re(\partial_{x}u)(t')  - GG^{*}v(t') - \frac{1}{2}\partial_{x}(v^{2})(t') \right\}}dt'  .
$$
Therefore, to find a solution to the system \eqref{NLS-KDV-Ex} in the class $\widetilde{X}^{s}_{T} \times \left( \widetilde{Y}^{s}_{T} \cap L^{2}(0, T; L^{2}_{0}(\mathbb{T})) \right)$ 
 is equivalent to proving the existence of a fixed point for the map 
$$
\Phi := (\Phi_{1}, \Phi_{2}): \widetilde{X}^{s}_{T} \times \left( \widetilde{Y}^{s}_{T} \cap L^{2}(0, T; L^{2}_{0}(\mathbb{T})) \right)\to \widetilde{X}^{s}_{T} \times \left( \widetilde{Y}^{s}_{T} \cap L^{2}(0, T; L^{2}_{0}(\mathbb{T})) \right)
$$
where
$$
\Phi_{1}(u, v) = e^{it\partial_{x}^{2}}u_{0} -i\int_{0}^{t}{e^{i(t-t')\partial_{x}^{2}}\left\{i\partial_{x}v(t') + \beta |u|^{2}u(t')  - i\varphi(t')^{2}a(x)^{2}u(t) \right\}}dt'
$$
and
\begin{equation*}
\Phi_{2}(u, v) =  e^{-t(\partial_{x}^{3} + \mu \partial_{x})}v_{0} +\int_{0}^{t}{e^{-(t-t')(\partial_{x}^{3} + \mu \partial_{x})}\left\{ Re(\partial_{x}u)(t') - \frac{1}{2}\partial_{x}(v^{2})(t') - GG^{*}v(t') \right\}}dt'. 
\end{equation*}

Our aim is to prove that $\Phi$ is a contraction in some ball of $\widetilde{X}^{s}_{T} \times \left( \widetilde{Y}^{s}_{T} \cap L^{2}(0, T; L^{2}_{0}(\mathbb{T})) \right)$. To do so, observe that the item $(i)$ of 
Proposition \ref{Est-Lin}, inequalities \eqref{Est-Tri} and \eqref{Est-Bi}, Proposition \ref{lema-bour} and estimate \eqref{Bou-misto1} give us
\begin{equation} \label{contr-u1}
\begin{split}
\|\Phi_{1}(u, v)\|_{\widetilde{X}^{s}_{T}} \lesssim&\ \|u_{0}\|_{H^{s}} + \|\partial_{x}v\|_{Z^{s}_{T}} + \| |u|^{2}u\|_{Z^{s}_{T}} + \|i\varphi(t')^{2}a(x)^{2}u\|_{Z^{s}_{T}} 
\\ \lesssim&\ \|u_{0}\|_{H^{s}} + \|\partial_{x}v\|_{Y^{s, \frac{1}{2}-}_{T}} + 
\|u\|_{X_T^{s, \frac{3}{8}}}^{3} + \|i\varphi(t')^{2}a(x)^{2}u\|_{X^{s, \frac{1}{2}}_{T}}   \\
 \lesssim &\ \|u_{0}\|_{H^{s}} +  T^{0+}\left(\|v\|_{Y_{T}^{s, \frac{1}{2}}} + \|u\|_{X_{T}^{s, \frac{1}{2}}}^{3} + \|u\|_{X_T^{s, \frac{1}{2}}} \right) 
\end{split}
\end{equation}
and, similarly,
\begin{equation} \label{contr-u2}
\begin{split}
\|\Phi_{1}(u,v) - \Phi_{1}(\tilde{u}, \tilde{v})\|_{\widetilde{X}_{T}^{s}} \lesssim& \ T^{0+}\|u - \tilde{u}\|_{\widetilde{X}_{T}^{s,\frac{1}{2}}}\left(1 + \|u\|_{X_T^{s, \frac{1}{2}}}^{2} + \|\tilde{u}\|_{X_{T}^{s, \frac{1}{2}}}^{2} \right)\\& + T^{0+}\|v - \tilde{v}\|_{\widetilde{Y}_{T}^{s}},
\end{split}
\end{equation}
for every $(u, v), (\tilde{u}, \tilde{v}) \in \widetilde{X}_T^{s} \times \left(\widetilde{Y}_T^{s} \cap L^{2}(0, T; L^{2}_{0}(\mathbb{T})\right)$.
On the other hand, observe that
$$
 \|Re(\partial_{x}u)\|_{\widetilde{Y}_T^{s}} = \left\| \frac{\partial_{x}u  + \overline{\partial_{x}u}}
 {2}\right\|_{\widetilde{Y}_T^{s}}\lesssim \|\partial_{x}u\|_{\widetilde{Y}_T^{s}},
$$
  then using items $(ii)$ and $(iv)$ of Proposition \ref{Est-Lin}, estimates \eqref{Est-Bi} and \eqref{Bou-misto2}, we have the following
\begin{equation} \label{contr-v1}
\begin{split}
\|\Phi_{2}(u, v)\|_{\widetilde{Y}_T^{s}} \lesssim& \ \|v_{0}\|_{H^{s}} + \|\partial_{x}u\|_{W^{s}_{T}} + \|\partial_{x}(v^{2})\|_{W^{s}_{T}} + \|GG^{*}v\|_{W^{s}_{T}} \\
\lesssim& \ \|v_{0}\|_{H^{s}}  + 
T^{0+}\left(\|u\|_{X_{T}^{s, \frac{1}{2}}} + \|v\|_{Y_{T}^{s, \frac{1}{2}}}^{2} \right) 
\end{split}
\end{equation}
and similarly, we get that,
\begin{equation} \label{contr-v2} 
\begin{split}
\|\Phi_{2}(u, v) - \Phi_{2}(\tilde{u}, \tilde{v})\|_{\widetilde{Y}_{T}^{s}} \lesssim &\ T^{0+}\|u - \tilde{u}\|_{\widetilde{X}_T^{s}} \\ &
+ T^{0+}\left(\|v\|_{Y_T^{s, \frac{1}{2}}} + \|\tilde{v}\|_{Y_T^{s, \frac{1}{2}}} \right)\|v - \tilde{v}\|_{\widetilde{Y}_T^{s}},  
\end{split}
\end{equation}
for every $(u, v), (\tilde{u}, \tilde{v}) \in \widetilde{X}_T^{s} \times \left(\widetilde{Y}_T^{s} \cap L^{2}(0, T; L^{2}_{0}(\mathbb{T})\right)$. 
We conclude from these estimates that if we take $T > 0$ small enough, the map $\Phi$ is a contraction in some suitable ball of $\widetilde{X}_T^{s} \times \left(\widetilde{Y}_T^{s} \cap L^{2}(0, T; L^{2}_{0}(\mathbb{T})\right)$, then it has a fixed point.

\vspace{0.2cm}

Now, we prove the uniqueness in the class $\widetilde{X}^{s}_{T} \times \widetilde{Y}^{s}_{T}$ for the integral equation $\Phi_1$ and $\Phi_2$. Set
$$
w(t) = e^{i t \partial_{x}^{2}} u_{0}-i \int_{0}^{t} e^{i\left(t-t^{\prime}\right) \partial_{x}^{2}}\left\{i \partial_{t} v\left(t^{\prime}\right)+\beta|u|^{2} u\left(t^{\prime}\right)-i \varphi\left(t^{\prime}\right)^{2} a(x)^{2} u\left(t^{\prime}\right)\right\} d t^{\prime}$$
and 
$$
z(t) = e^{-t\left(\partial_{x}^{3}+\mu \partial_{x}\right)} v_{0}+\int_{0}^{t} e^{-\left(t-t^{\prime}\right)\left(\partial_{x}^{3}+\mu \partial_{x}\right)}\left\{\operatorname{Re}\left(\partial_{x} u\right)\left(t^{\prime}\right)-GG^{*} v\left(t^{\prime}\right)-\frac{1}{2} \partial_{x}\left(v^{2}\right)\left(t^{\prime}\right)\right\} d t^{\prime}.
$$
Observe that, $\beta|u|^{2}u, \partial_{x}v \in X^{s, -\frac{1}{2}}_{T}$ and $\frac{1}{2}\partial_{x}(v^{2}), \partial_{x}u \in Y^{s, -\frac{1}{2}}_{T}$, hence we infer that,
\begin{equation*}
\begin{split}
\partial_{t}\left(\int_{0}^{t} e^{-it^{\prime} \partial_{x}^{2}}\left\{i \partial_{x} v+\beta|u|^{2} u-i \varphi^{2} a(x)^{2} u\right\}\left(t^{\prime}\right) d t^{\prime} \right) = e^{-it \partial_{x}^{2}}\left( i \partial_{x} v+\beta|u|^{2}u -i \varphi^{2} a^{2} u\right)
\end{split}
\end{equation*}
and
\begin{equation*}
\begin{split}
\partial_{t}&\left(\int_{0}^{t} e^{-t^{\prime}\left(\partial_{x}^{3}+\mu \partial_{x}\right)}\left\{\operatorname{Re}\left(\partial_{x} u\right)-GG^{*}v-\frac{1}{2}\partial_{x}\left(v^{2}\right)\right\}\left(t^{\prime}\right) d t^{\prime} \right) \\&= e^{-t(\partial_{x}^{3} + \mu\partial_{x})}\left(\operatorname{Re}\left(\partial_{x} u\right)-GG^{*}v-\frac{1}{2}\partial_{x}\left(v^{2}\right) \right),
\end{split}
\end{equation*}
in the distributional sense. This implies that $(w, z)$ is a solution of 
 \begin{equation*}
\left\{\begin{array}{ll}
 i \partial_{t}w + \partial_{x}^{2}w = i \partial_{x} v+\beta|u|^{2}u  -i \varphi^{2} a(\cdot)^{2} u,& (x,t)\in\mathbb{T} \times \mathbb{R}_+,\\
\partial_{t}z + \partial_{x}^{3}z + \mu \partial_{x}z = \operatorname{Re}\left(\partial_{x} u\right)-GG^{*}v-\frac{1}{2}\partial_{x}\left(v^{2}\right),& (x,t)\in\mathbb{T} \times \mathbb{R}_+.
\end{array}\right.
\end{equation*}
Therefore, it follows that, $r_{1}(t) = e^{-it\partial^2_x}(u-w)$  and $r_{2}(t) = e^{-t(\partial^3_x+\mu\partial_x)}(v-z)$ is a solution of  $\partial_{t}r_{1} = \partial_{t}r_{2} = 0$ and $r_{1}(t) = r_{2}(t) = 0$. Thus $r_{1} = r_{2} = 0$ and $(u,v)$ is the unique solution of the integral equations.
 
 \begin{itemize}
\item[(b)]\textbf{The solution is globally defined.}
\end{itemize}

 In order to prove that our solutions are global, we observe that, 
\begin{equation*}
\begin{split}
 \frac{d}{dt}\left(\|u(t)\|_{L^{2}(\mathbb{T})}^{2} + \|v(t)\|_{L^{2}_{0}(\mathbb{T})}^{2} \right) \leq & 0.
\end{split}
\end{equation*}
 Then, integrating over $t$ and using the Gronwall Inequality we get that, the $L^{2}(\mathbb{T}) \times L^{2}_{0}(\mathbb{T})$ norm of $(u,v)$ remains bounded by every $t \in [0, T]$, hence the local solution of \eqref{NLS-KDV-Ex} can be extended to a global one.
 
Now we prove that, if $(u_{0}, v_{0}) \in H^{2}(\mathbb{T}) \times H^{3}_{0}(\mathbb{T})$ then the solution $(u,v) \in C([0, T]; H^{2}(\mathbb{T})\times H^{2}_{0}(\mathbb{T}))$ of \eqref{NLS-KDV-Ex} can be extended for any $T > 0$ and then, by nonlinear interpolation  (see, e.g., \cite{Bona} and \cite{Tartar}), we can get the global well-posedness for the solution $(u,v) \in C([0, T]; H^{s}(\mathbb{T})\times H^{s}_0(\mathbb{T}))$ of \eqref{NLS-KDV-Ex} with $(u_{0}, v_{0}) \in H^{s}(\mathbb{T}) \times H^{s}_{0}(\mathbb{T})$ with $0 \leq s \leq 2$. 

Indeed, we start by considering a smooth solution $(u,v)$ of 
\eqref{NLS-KDV-Ex}. Let $(z, w) = (u_{t}, v_{t})$ so,  
  \begin{equation}\label{NLS-KDV-Ex-zw}
\left\{\begin{array}{ll}
i \partial_{t} z+\partial_{x}^{2} z + ia(x)^{2}z = i \partial_{x}w + 2\beta|u|^{2} z + \beta u^{2}\overline{z}, & (x,t)\in\mathbb{T} \times \mathbb{R}_{+},\\
\partial_{t} w+\partial_{x}^{3} w+ \partial_{x}\left(v w\right)+ \mu\partial_{x}w = Re(\partial_{x}z) - GG^{*}w, & (x,t)\in\mathbb{T} \times \mathbb{R}_{+},\\
z(x, 0)= z_{0}(x), \quad w(x, 0)= w_{0}(x), & x\in\mathbb{T},
\end{array}\right.
\end{equation}
where $$z_{0} = iu_{0xx} + v_{0x} - i\beta|u_{0}|^{2}u_{0} - a^{2}u_{0} \in L^{2}(\mathbb{T})\in L^{2}(\mathbb{T})$$ and $$w_{0} = -v_{0xxx} - \mu v_{0x} - v_{0}v_{0x} + Re(u_{0x}) - GG^{*}v_{0} \in L^{2}_{0}(\mathbb{T}).$$ If we consider $T > 0$ such that $\Phi$ is a contraction in some suitable ball of $\widetilde{X}^{0}_{T} \times \widetilde{Y}^{0}_{T}$ as in \eqref{contr-u2} and \eqref{contr-v2}, we have that 
\begin{eqnarray*}
\|(u,v)\|_{\widetilde{X}^{0}_{T} \times \widetilde{Y}^{0}_{T}} \leq C_{1},
\end{eqnarray*}
with $C_{1} = C_{1}(\|u_{0}\|_{L^{2}(\mathbb{T})}, \|v_{0}\|_{L^{2}_{0}(\mathbb{T})}) > 0$. The same computation as in \eqref{contr-u1} and \eqref{contr-u2} leads us to 
 \begin{eqnarray*}
 \|z\|_{\widetilde{X}^{0}_{T}} \lesssim \|z_{0}\|_{L^{2}(\mathbb{T})} + T^{0+} \left(\|w\|_{\widetilde{Y}^{0}_{T}} + \|z\|_{\widetilde{X}^{0}_{T}} + 2C_{1}^{2}\|z\|_{\widetilde{X}^{0}_{T}} \right) 
 \end{eqnarray*}
 and 
 \begin{eqnarray*}
 \|w\|_{\widetilde{Y}^{0}_{T}} \lesssim \|w_{0}\|_{L^{2}_{0}(\mathbb{T})} + T^{0+}\left(\|z\|_{\widetilde{X}^{0}_{T}} + C_{1}\|w\|_{\widetilde{Y}^{0}_{T}} \right),
 \end{eqnarray*}
or equivalently, 
 \begin{eqnarray*}
 \|(z,w)\|_{\widetilde{X}^{0}_{T} \times \widetilde{Y}^{0}_{T}} \leq C\|(z_{0}, w_{0})\|_{L^{2}(\mathbb{T}) \times L^{2}_{0}(\mathbb{T})}.
 \end{eqnarray*}
Therefore, using the previous inequalities, we have
 \begin{eqnarray*}
 \|(z, w)\|_{L^{\infty}(0, T; L^{2}(\mathbb{T}) \times L^{2}_{0}(\mathbb{T}))} \leq \|(z,w)\|_{\widetilde{X}^{0}_{T} \times \widetilde{Y}^{0}_{T}} \leq C\|(z_{0}, w_{0})\|_{L^{2}(\mathbb{T}) \times L^{2}_{0}(\mathbb{T})}.
 \end{eqnarray*}
Thanks to the system \eqref{NLS-KDV-Ex}, we get that
\begin{equation}  \label{inter}
\begin{split}
\|\partial_{x}^{3}v\|_{L^{2}_0(\mathbb{T})} \leq& \|w\|_{L^{2}_0(\mathbb{T})} + \|\operatorname{Re}(\partial_{x}u)\|_{L^{2}(\mathbb{T})} + \frac{1}{2} \|v\partial_{x}v\|_{L^{2}_0(\mathbb{T})} + 
|\mu|\|\partial_{x}v\|_{L^{2}_0(\mathbb{T})} + \|GG^{*}v\|_{L^{2}_0(\mathbb{T})} \\
\leq& \|w\|_{L^{2}_0(\mathbb{T})} + C\|\partial_{x}u\|_{L^{\infty}(\mathbb{T})} + \|v\|_{L^{2}_0(\mathbb{T})}\|\partial_{x}v\|_{L^{\infty}(\mathbb{T})} + \|v\|_{L^{2}_0(\mathbb{T})} +
|\mu|\|\partial_{x}v\|_{L^{2}_0(\mathbb{T})}  \\ 
 \leq & \|w\|_{L^{2}_0(\mathbb{T})} + C\|u\|_{L^{2}(\mathbb{T})}^{\frac{1}{2}}\|\partial_{x}^{2}u\|_{L^{2}(\mathbb{T})}^{\frac{1}{2}} + C(1 + \|v\|_{L^{2}_0(\mathbb{T})})\|v\|_{L^{2}_0(\mathbb{T})}^{\frac{1}{2}}\|\partial_{x}^{3}v\|_{L^{2}_0(\mathbb{T})}^{\frac{1}{2}} + 
|\mu|\|\partial_{x}v\|_{L^{2}_0(\mathbb{T})}\\
 \leq & \frac{1}{4}\|(\partial_{x}^{2}u, \partial_{x}^{3}v)\|_{L^{2}(\mathbb{T}) \times L^{2}_{0}(\mathbb{T})} + \|w\|_{L^{2}_0(\mathbb{T})} + C(\|v\|_{L^{2}_0(\mathbb{T})} + \|v\|_{L^{2}_0(\mathbb{T})}^{3} + \|u\|_{L^{2}(\mathbb{T})}),
\end{split}
\end{equation}
where we have used the Young inequality in this last inequality and
\begin{equation}  \label{inter-1}
\begin{split}
 \|\partial_{x}^{2}u\|_{L^{2}(\mathbb{T})}  \leq&  \|z\|_{L^{2}(\mathbb{T})} + \|\partial_{x}v\|_{L^{2}_0(\mathbb{T})} + |\beta|\|u\|_{L^{2}(\mathbb{T})}\|u\|_{L^{\infty}}^{2} + \|a\|_{L^{\infty}}\|u\|_{L^{2}(\mathbb{T})} \\ 
 \leq & \frac{1}{4}\|(\partial_{x}^{2}u, \partial_{x}^{3}v)\|_{L^{2}(\mathbb{T}) \times L^{2}_{0}(\mathbb{T})} + \|z\|_{L^{2}(\mathbb{T})} + C|\beta|\|u\|_{L^{2}(\mathbb{T})}^{2} + \|a\|_{L^{\infty}(\mathbb{T})}\|u\|_{L^{2}(\mathbb{T})} ,
\end{split}
\end{equation}
where, for this last inequality, we have used that $H^{2}(\mathbb{T})$ is embedded in $L^{\infty}(\mathbb{T})$, the interpolation between $H^{2}(\mathbb{T})$ and $L^{2}(\mathbb{T})$ and Young inequality. So adding up the estimates \eqref{inter} and 
 \eqref{inter-1} we can conclude that 
\begin{eqnarray*}
 \|(\partial_{x}^{2}u, \partial_{x}^{3}v)(t)\|_{L^{2}(\mathbb{T}) \times L^{2}_{0}(\mathbb{T})} \leq C(\|u_{0}\|_{L^{2}(\mathbb{T})}, \|v_{0}\|_{L^{2}_0(\mathbb{T})})\|(u_{0}, v_{0})\|_{H^{2}(\mathbb{T})\times H^{3}_0(\mathbb{T})},
\end{eqnarray*}
 from which follows that $(u,v)\in C(\mathbb{R}_{+}; H^{2}(\mathbb{T}) \times H^{2}_{0}(\mathbb{T}))$. A similar argumentation can be done for $(u_{0}, v_{0}) \in H^{2k}(\mathbb{T}) \times H^{3k}_0(\mathbb{T})$, $k \in \mathbb{N}$, and for other values of $s \in \mathbb{R}_{+}$ using again nonlinear interpolation.


\begin{itemize}
\item[(c)]\textbf{The flow map is Lipschitz.}
\end{itemize}
 
 Finally we prove that the map \eqref{flow}  is Lipschitz on bounded sets. To do so, 
 consider $(u,v)$ and $(\widetilde{u}, \widetilde{v})$ solutions of \eqref{NLS-KDV-Ex} with initial data $(u_{0}, v_{0}), (\widetilde{u_{0}}, \widetilde{v_{0}})$, respectively. Arguing as in \eqref{contr-u2} and \eqref{contr-v2} we have that,
\begin{equation*}
\begin{split}
 \|u - \widetilde{u}\|_{\widetilde{X}^{0}_{T}} + 
 \|v - \widetilde{v}\|_{\widetilde{Y}^{0}_{T}} \leq& \ C\|u_{0} - \widetilde{u}_{0}\|_{L^{2}(\mathbb{T})} + C\|v_{0} - \widetilde{v}_{0}\|_{L^{2}_0(\mathbb{T})} 
 \\& 
 + T^{0+}\left(1 + \|u\|_{\widetilde{X}^{0}_{T}} + \|\widetilde{u}\|_{\widetilde{X}^{0}_{T}}\right) \|u-\widetilde{u}\|_{\widetilde{X}^{0}_{T}}\\& + CT^{0+}\left( 1 + \|v\|_{\widetilde{Y}_{T}^{0}} 
 + \|\widetilde{v}\|_{\widetilde{Y}_{T}^{0}}\right)\|v - \widetilde{v}\|_{\widetilde{Y}^{0}_{T}}.
\end{split}
\end{equation*}
  Then, for $T > 0$ small enough, 
  depending of the size of $(u_{0}, v_{0}), (\widetilde{u}_{0}, \widetilde{v}_{0})$,
  it follows that 
\begin{equation}\label{Lips-Flow}
\begin{split}
 \|u - \widetilde{u}\|_{\widetilde{X}^{0}_{T}} + 
 \|v - \widetilde{v}\|_{\widetilde{Y}^{0}_{T}} \leq& \ C\|u_{0} - \widetilde{u}_{0}\|_{L^{2}(\mathbb{T})} + C\|v_{0} - \widetilde{v}_{0}\|_{L^{2}_0(\mathbb{T})}.
\end{split}
\end{equation}
By iterating process we get that \eqref{Lips-Flow} is valid for every $T$, concluding the proof of Theorem \ref{Teo-Ex}. 
 \end{proof}

  \section{Stabilization result}\label{Sec4}
In this section we are able to prove one of the main results of the article, precisely, we prove that the following system 
 \begin{equation}\label{est-prob}
\left\{\begin{array}{ll}
i \partial_{t} u+\partial_{x}^{2} u + ia(x)^{2}u = i \partial_{x}v+\beta|u|^{2} u ,  & (x,t)\in\mathbb{T} \times \mathbb{R}_+,\\
\partial_{t} v+\partial_{x}^{3} v+\frac{1}{2} \partial_{x}\left(v^{2}\right) + \mu\partial_{x}v = Re(\partial_{x}u) - GG^{*}v,  & (x,t)\in\mathbb{T} \times\mathbb{R}_+,\\
u(x, 0)=u_{0}(x), \quad v(x, 0)=v_{0}(x), & x\in\mathbb{T},
\end{array}\right.
\end{equation}
where $\beta,\mu\in\mathbb{R}$, is asymptotically stable for $(u_{0}, v_{0}) \in L^{2}(\mathbb{T}) \times L^{2}_{0}(\mathbb{T}) $, when two control inputs are acting in both equation in $\omega\subset\mathbb{T}$.

\subsection{Proof of Theorem \ref{est-teo_int}} As usual in the literature using the \textit{``Compactness–Uniqueness Argument"} due Lions \cite{Lions}, under the hypothesis of the Theorem \ref{est-teo_int}, the global stabilization property is equivalent to show the following observability inequality: 

\vspace{0.2cm}

\textit{For any $T > 0$ there exists $C = C(T) > 0$ such that
\begin{eqnarray}\label{obs}
 \|(u_{0}, v_{0})\|_{L^{2}(\mathbb{T}) \times L^{2}_{0}(\mathbb{T})}^{2} 
 \leq C \left(\int_{0}^{T}{\|au(t)\|_{L^{2}(\mathbb{T})}^{2}}dt + \int_{0}^{T}{\|Gv(t)\|_{L^{2}_{0}(\mathbb{T})}^{2}}dt \right),
\end{eqnarray}
for any solution $(u, v)$ of the system \eqref{est-prob} with initial data $(u_{0}, v_{0}) \in L^{2}(\mathbb{T}) \times L^{2}_{0}(\mathbb{T})$ such that $\|(u_{0}, v_{0})\|_{L^{2}(\mathbb{T}) \times L^{2}_{0}(\mathbb{T})} \leq R_{0}$. }

\vspace{0.2cm}

Indeed, it follows from the energy estimate \eqref{derivative} that 
$$
\|(u, v)(T)\|_{L^{2}(\mathbb{T}) \times L^{2}_{0}(\mathbb{T})}^{2} = \|(u_{0}, v_{0})\|_{L^{2}(\mathbb{T}) \times L^{2}_{0}(\mathbb{T})}^{2} - \int_{0}^{T}{\|au(t)\|_{L^{2}(\mathbb{T})}^{2}}dt - \int_{0}^{T}{\|Gv(t)\|_{L^{2}_{0}(\mathbb{T})}^{2}}dt.
$$
Thus, from \eqref{obs}, we get that
$$
\|(u, v)(T)\|_{L^{2}(\mathbb{T}) \times L^{2}_{0}(\mathbb{T})}^{2} \leq (1 - C^{-1})\|(u_{0}, v_{0})\|_{L^{2}(\mathbb{T}) \times L^{2}_{0}(\mathbb{T})}^{2}, 
$$
and since the solution $(u,v)$ of the system \eqref{est-prob} satisfies the semigroup property we have, for every $m \in \mathbb{N}$, that
$$
\|(u, v)(mT)\|_{L^{2}(\mathbb{T}) \times L^{2}_{0}(\mathbb{T})}^{2} \leq (1 - C^{-1})^{m}\|(u_{0}, v_{0})\|_{L^{2}(\mathbb{T}) \times L^{2}_{0}(\mathbb{T})}^{2},
$$
which yields the desired result. \qed

\subsection{Proof of the observability inequality}
We argue by contradiction. If \eqref{obs} does not occur, there exist $T > 0$ 
and a sequence $\{(u_{0n}, v_{0n})\}_{n\in\mathbb{N}}$ such that
\begin{equation}\label{contrad1} 
 \|(u_{0n}, v_{0n})\|_{L^{2}(\mathbb{T}) \times L^{2}_{0}(\mathbb{T})} \leq R_{0} 
\end{equation}
and
\begin{equation}\label{contrad2}
 \int_{0}^{T}{\|au_{n}(t)\|_{L^{2}(\mathbb{T})}^{2}}dt + \int_{0}^{T}{\|Gv_{n}(t)\|_{L^{2}_{0}(\mathbb{T})}^{2}}dt \leq \frac{1}{n} \|(u_{0n}, v_{0n})\|_{L^{2}(\mathbb{T}) \times L^{2}_{0}(\mathbb{T})}^{2}. 
\end{equation}
Hence if we define $\alpha_{n} := \|(u_{0n}, v_{0n})\|_{L^{2}(\mathbb{T}) \times L^{2}_{0}(\mathbb{T})}^{2}$ we get  $$\alpha_{n} \leq R_{0}^{2},$$ so we can extract a subsequence, still denoted by the same index, such that 
$$\alpha_{n} \to \alpha,$$ with $\alpha \geq 0$.  We split the analysis into two cases: 
$$(i)\ \alpha > 0 \quad \text{and} \quad (ii)\ \alpha = 0.$$

\begin{itemize}\item[$(i)$] $\alpha > 0$.\end{itemize} 

It follows from Theorem $\ref{Teo-Ex}$ that the corresponding sequence of solutions $\{(u_{n}, v_{n})\}$ associated with 
the initial data $\{(u_{0n}, v_{0n})\}$ is bounded in both spaces 
$L^{\infty}(0, T; L^{2}(\mathbb{T}) \times L^{2}_{0}(\mathbb{T}))$ 
and $X_{T}^{0, \frac{1}{2}} \times Y_{T}^{0, \frac{1}{2}}$ then as $X_{T}^{0, \frac{1}{2}} \times Y_{T}^{0, \frac{1}{2}}$ is a separable Hilbert space compactly embedded in $X_{T}^{- k, \frac{1}{2} - \delta} \times Y_{T}^{ -s, \frac{1}{2}-\delta_{1}}$, for every $k, s > 0$ and $\delta, \delta_{1} > 0$, we can get a subsequence, which will be denoted with the same index, such that 
$$(u_{n}, v_{n}) \to (u,v) \text{ weakly in }X_{T}^{0, \frac{1}{2}} \times Y_{T}^{0, \frac{1}{2}}$$
and 
$$ (u_{n}, v_{n}) \to (u,v) \text{ strongly in } X_{T}^{-k,  \frac{1}{2}-\delta} \times Y_{T}^{-s, \frac{1}{2} - \delta_{1}}.$$
Moreover, by \eqref{contrad2} we deduce that
$$au_{n} \to 0 \quad \hbox{ in } L^{2}(0, T; L^{2}(\mathbb{T}))$$
and 
$$Gv_{n} \to 0 = Gv \quad \hbox{ in } L^{2}(0, T; L^{2}_{0}(\mathbb{T})),$$
from which follows that $u \equiv 0$ in $\omega$ and $v = c(t)$ in $\omega$. 

On the other hand, by Proposition \ref{Est-Mul}, we infer that  $\beta |u_{n}|^{2}u_{n}$ is bounded in $X_{T}^{0, -\frac{1}{2}}$ and $\frac{1}{2}\partial_{x}(v^{2}_{n})$ is bounded in  $Y^{0, -\frac{1}{2}}_{T}$. Moreover, note that, $(\beta |u_{n}|^{2}u_{n})$ is bounded in $X^{0, -b'}_{T}$ for $\frac{3}{8} < b' < \frac{1}{2}$, which is compactly embedded in $X_{T}^{-1, -\frac{1}{2}}$, then we can extract a subsequence, still denoted by the same index, such that, for some $f \in X_{T}^{-1, -\frac{1}{2}}$,
\begin{eqnarray*}
\beta |u_{n}|^{2}u_{n} \to f \quad \hbox{ strongly in } X^{-1, -\frac{1}{2}}_{T}.
\end{eqnarray*}
As $Y_{T}^{0, \frac{1}{2}}$ is continuously embedded in $L^{4}(\mathbb{T} \times (0, T))$, thanks to the Proposition \ref{St-KdV}, $(\partial_{x}(v_{n}^{2}))$ is bounded in $L^{2}(0, T; H^{-1}(\mathbb{T})) = Y_{T}^{-1, 0}$, so, by interpolation between 
those spaces,  we get that $(\partial_{x}(v_{n}^{2}))$ is bounded in $Y_{T}^{- \theta, -\frac{1}{2} + \frac{\theta}{2}}$, $\theta \in (0,1)$, which is compactly embedded in $Y_{T}^{-1, -\frac{1}{2}}$. Therefore, we can extract a subsequence of $\left(\frac{1}{2}\partial_{x}(v_{n}^{2})\right)$ in $Y_{T}^{-1, -\frac{1}{2}}$, which will be denoted by the same index, such that
$$\frac{1}{2}\partial_{x}(v_{n}^{2}) \to \tilde{g} \quad \hbox{ strongly in } Y_{T}^{-1, -\frac{1}{2}}. $$
With these convergences in hand, we can pass to the limit in $n$, to obtain  that $(u,v)$ satisfies
 \begin{equation*}
\left\{\begin{array}{ll}
i\partial_{t}u + \partial_{x}^{2}u = i\partial_{x}v + f,  & (x,t)\in\mathbb{T} \times (0,T),\\
\partial_{t}v + \partial_{x}^{3}v + \mu\partial_{x}v = Re(\partial_{x}u) + \tilde{g},  & (x,t)\in\mathbb{T} \times (0,T).
\end{array}\right.
\end{equation*}

Now, consider  $w_{n} = u_{n} - u$, $z_{n} = v_{n} - v$, $$f_{n} = -ia^{2}u_{n} +\beta |u_{n}|^{2}u_{n} - f$$ and $$g_{n} = -GG^{*}v_{n} - \frac{1}{2}\partial_{x}(v_{n}^{2}) -  \tilde{g}.$$ Observe that  
$$
\int_{0}^{T}{\|Gz_{n}\|_{L^{2}_0(\mathbb{T})}^{2}}dt = 
\int_{0}^{T}{\|Gv_{n}\|_{L^{2}_0(\mathbb{T})}^{2}}dt + \int_{0}^{T}{\|Gv\|_{L^{2}_0(\mathbb{T})}^{2}}dt - 2\int_{0}^{T}{(Gv_{n}, Gv)_{0}}dt \to 0.
$$
On the other hand, we have
$$
Gz_{n}(x,t) = g(x)z_{n}(x, t) + \int_{\mathbb{T}}{g(y)z_{n}(y, t)}dy.
$$
Since $z_{n} \to 0$ weakly in $Y^{0, \frac{1}{2}}_{T}$, using Rellich theorem we get that 
$\int_{\mathbb{T}}{g(y)z_{n}(y, \cdot)}dy$ strongly converges to $0$ in $L^{2}(0, T)$, from which follows that
$$
\int_{0}^{T}{\int_{\mathbb{T}}}{g^{2}(x)z_{n}(x,t)^{2}}dxdt \to 0.
$$
Thus, 
 \begin{equation*}
\left\{\begin{array}{ll}
i\partial_{t}w_{n} + \partial_{x}^{2}w_{n} = i\partial_{x}z_{n} + f_{n},  & (x,t)\in\mathbb{T} \times (0,T),\\
\partial_{t}z_{n} + \partial_{x}^{3}z_{n} + \mu \partial_{x}z_{n} = Re(\partial_{x}w_{n}) + g_{n},  & (x,t)\in\mathbb{T} \times (0,T)
\end{array}\right.
\end{equation*}
with 
$$
(w_{n}, z_{n}) \to (0, 0) \quad \hbox{ in } L^{2}(0, T; L^{2}(\widetilde{\omega}) \times L^{2}_{0}(\widetilde{\omega}))
$$
and
$$
(f_{n}, g_{n}) \to (0, 0) \quad \hbox{ strongly in } X^{-1, -\frac{1}{2}}_{T} \times Y^{-1, -\frac{1}{2}}_{T},
$$
where $\widetilde{\omega} = supp(g)\cap supp(a)$.

Now, we are in a position to use the results of Appendix \ref{Appendix1} and \ref{Appendix2}. First, using the  propagation of compactness, given in Proposition \ref{Prop-Comp}, we have 
$$
(z_{n}, w_{n}) \to (0, 0) \quad \hbox{ in } L^{2}_{loc}(0, T; L^{2}(\mathbb{T}) \times L^{2}_{0}(\mathbb{T})).
$$
Then we can pick $t_{0} \in [0, T]$  such that $(z_{n}(t_{0}), w_{n}(t_{0})) \to (0,0)$ strongly in $L^{2}(\mathbb{T})\times L^{2}_{0}(\mathbb{T})$. Denote by $(\widetilde{u}, \widetilde{v})$ the solution of the problem
 \begin{equation*}
\left\{\begin{array}{ll}
i \partial_{t}\widetilde{u} + \partial_{x}^{2}\widetilde{u} = i \partial_{x}\widetilde{v} + \beta |\widetilde{u}|^{2}\widetilde{u},  & (x,t)\in\mathbb{T} \times (0,T),\\
\partial_{t}\widetilde{v} + \partial_{x}^{3}\widetilde{v} + \mu \partial_{x}\widetilde{v} + \partial_{x}(\widetilde{v}^{2}) = Re(\partial_{x}\widetilde{u}),  & (x,t)\in\mathbb{T} \times (0,T),\\
(\widetilde{u}, \widetilde{v})(t_{0}) = (u(t_{0}), v(t_{0})), & t_0\in(0,T).
\end{array}\right.
\end{equation*}
Theorem \ref{Teo-Ex} gives us that the flow map is Lipschitz on bounded sets, hence as 
$$(u_{n}(t_{0}), v_{n}(t_{0})) \to (u(t_{0}), v(t_{0})) \quad \text{in } L^{2}(\mathbb{T})\times L^{2}_0(\mathbb{T}),$$ 
$$i a^{2}u_{n} \to 0 \quad \text{in } L^{2}(0, T; L^{2}(\mathbb{T}))$$ and $$Gv_{n} \to 0 \quad \text{in } L^{2}(0, T; L^{2}_0(\mathbb{T})),$$follows that $$(u_{n}, v_{n}) \to (u, v)\quad \text{in}\quad  
X_{T}^{0, \frac{1}{2}} \times Y_{T}^{0, \frac{1}{2}}.$$ 
Therefore,  putting together all these convergences and passing to the limit, we can conclude that $(u,v)$ solves
\begin{equation}\label{4.5}
\left\{\begin{array}{ll}
i \partial_{t} u+\partial_{x}^{2} u = i \partial_{x} v+\beta|u|^{2} u,  & (x,t)\in\mathbb{T} \times (0,T),\\
\partial_{t} v+\partial_{x}^{3} v+\frac{1}{2} \partial_{x}\left(v^{2}\right)+\mu \partial_{x} v=Re\left(\partial_{x} u\right),  & (x,t)\in\mathbb{T} \times (0,T),\\
u(x,t) = 0 \quad v(x,t)= c(t),  & (x,t)\in\omega \times (0, T),\\
u(x, 0)=u_{0}(x), \quad v(x, 0)=v_{0}(x), &x\in\mathbb{T}.
\end{array}\right.
\end{equation}
Note that, using the second equation of \eqref{4.5} we have that $\partial_{t}v = 0$ on $\omega \times (0, T)$, hence $c(t) \equiv c$,  with $c \in \mathbb{R}$, for all $t \in (0, T)$. Thus thanks to the unique continuation property given by Corollary \ref{Uniq-Cont-Cor}, we ensure that $(u,v) = (0, 0)$.  From this, we conclude that 
$$
\|(u_{n}(0), v_{n}(0))\|_{L^{2}(\mathbb{T}) \times L^{2}_{0}(\mathbb{T})} \to 0,
$$
which is a contradiction with our hypothesis $\alpha > 0$.

\begin{itemize}\item[$(ii)$] $\alpha = 0.$\end{itemize}

Observe that we can assume  $\alpha_{n} > 0$ for all $n \in \mathbb{N}$. Pick the function as follows $w_{n} = \frac{u_{n}}{\alpha_{n}}$ and 
$z_{n} = \frac{v_{n}}{\alpha_{n}}$, for all $n \geq 1$. Thus,
$$
\|(w_{0n}, z_{0n})\|_{L^{2}(\mathbb{T}) \times L^{2}_{0}(\mathbb{T})}^{2} = 1
$$
and $(w_{n}, z_{n})$ satisfies the system 
\begin{equation*}
\left\{\begin{array}{ll}
 i \partial_{t}w_{n} + \partial_{x}^{2}w_{n} = i\partial_{x} z_{n} + \alpha_{n}^{2}\beta|v_{n}|^{2}v_{n} - ia(x)^{2}w_{n}, & (x,t)\in\mathbb{T} \times (0,T),\\
\partial_{t}z_{n} + \partial_{x}^{3}z_{n} + \frac{\alpha_{n}^{2}}{2}\partial_{x}(v^{2}_{n}) = Re(\partial_{x}w_{n}) - GG^{*}z_{n},  & (x,t)\in\mathbb{T} \times (0,T)
\end{array}\right.
\end{equation*}
with 
$$
 \int_{0}^{T}{\|a(x)w_{n}(t)\|_{L^{2}(\mathbb{T})}^{2}}dt + \int_{0}^{T}{\|Gz_{n}(t)\|_{L^{2}_{0}(\mathbb{T})}^{2}}dt \leq \frac{1}{n}, \quad\forall n \in \mathbb{N}. 
$$
By a boot-strap argument we conclude that $(w_{n}, z_{n})$ 
is bounded in $X^{0, \frac{1}{2}}_{T} \times Y_{T}^{0, \frac{1}{2}}$. Hence we can extract a subsequence of $\{(w_{n}, z_{n}) \}$, still denoted by the same index,  such that 
$$ (w_{n}, z_{n}) \to (w,z) \quad  \text{weakly in} \quad X_{T}^{0, \frac{1}{2}} \times Y_{T}^{0, \frac{1}{2}}$$ and
$$ (w_{n}, z_{n}) \to (w,z) \quad  \text{strongly in} \quad X_{T}^{0-, -\frac{1}{2}} \times Y_{T}^{0-, -\frac{1}{2}}.$$ By Proposition \ref{Est-Mul} 
we have that $\beta |w_{n}|^{2}w_{n}$ is bounded in 
$X_{T}^{0, -\frac{1}{2}}$ then $$\alpha_{n}^{2}\beta |w_{n}|^{2}w_{n} \to 0 \quad \text{strongly in }\quad X_{T}^{0, -\frac{1}{2}},$$
 as $\alpha_{n} \to 0$.  Similarly, due to the fact that  $\{\partial_{x}(v_{n}^{2}) \}$ is bounded in $Y_{T}^{0, -\frac{1}{2}}$ we have that $$\alpha_{n}\partial_{x}(v_{n}^{2}) \to 0 \quad \text{strongly in } \quad Y_{T}^{0, -\frac{1}{2}}.$$ Hence, $(w, z)$ solves the following system 
\begin{equation*}
\left\{\begin{array}{ll}
 i \partial_{t}w + \partial_{x}^{2}w = i \partial_{x}z, & (x,t)\in\mathbb{T} \times (0,T),\\
\partial_{t}z + \partial_{x}^{3}z + \mu \partial_{x}z= Re(\partial_{z}w), & (x,t)\in\mathbb{T} \times (0,T),\\
w(x,t) = 0, \quad z(x,t) =c,& (x,t)\in \omega\times(0,T).
\end{array}\right.
\end{equation*}
Again using Corollary $\ref{Uniq-Cont-Cor}$  we have that  $w = z = 0$ in $(0, T) \times \mathbb{T}$. Then, by an application of Proposition \ref{Prop-Comp}, 
$$
(w_{n}, z_{n}) \to (0,0) \quad \text{ in  } \quad L^{2}_{loc}(0, T; L^{2}(\mathbb{T})\times L^{2}_{0}(\mathbb{T})).
$$
So we can conclude the proof as in the first case, showing the observability inequality \eqref{obs}. \qed 

\section{Controllability results }\label{Sec5}
In this section, we are interested to prove the exact controllability for the following nonlinear system \begin{equation}\label{prob-contr}
\left\{\begin{array}{ll}
i \partial_{t} u+\partial_{x}^{2} u = i \partial_{x} v+\beta|u|^{2} u + f, & (x,t)\in\mathbb{T} \times (0,T),\\
\partial_{t} v+\partial_{x}^{3} v + \frac{1}{2} \partial_{x}\left(v^{2}\right) + \mu\partial_{x}v = \operatorname{Re}\left(\partial_{x} u\right) + Gh,  & (x,t)\in \mathbb{T} \times (0,T),\\
u(x, 0)=u_{0}(x), \quad v(x, 0)=v_{0}(x), & x \in \mathbb{T},
\end{array}\right.
\end{equation}
where $f$ and $h$ are control functions. Before presenting the result for this system, we first need to prove the result for the linear system associated with \eqref{prob-contr},
\begin{equation}\label{prob-contr_lin}
\left\{\begin{array}{ll}
i \partial_{t} u+\partial_{x}^{2} u =  
f, & (x,t)\in\mathbb{T} \times (0,T),\\
\partial_{t} v+\partial_{x}^{3} v + \mu\partial_{x}v = 
Gh,  & (x,t)\in \mathbb{T} \times (0,T),\\
u(x, T) = 0, \quad v(x, T) = 0, & x \in \mathbb{T}.
\end{array}\right.
\end{equation}
In both cases we take $f$ and $h$ with a special form, that is, 
\begin{equation}\label{fg}
f:=a^2\varphi^2h_{1} \quad \text{ and } \quad h:=G^{*}h_{2}=Gh_2,
\end{equation} 
with $a$ and $G$ satisfying \eqref{Def-a} and \eqref{Def-G}, respectively, and $h_{1},h_{2}$ in some appropriated space.

It is classical in the literature that the observability inequality for the NLS equation holds, precisely, the following proposition related with the observability inequality for the adjoint system associated to Schrödinguer equation is verified and for details see, e.g. \cite{Lebeau} and \cite{Machtyngier}.
\begin{proposition}
Let $\omega \subset \mathbb{T}$. 
For any $a = a(x) \in C^{\infty}(\mathbb{T})$ and $\varphi = \varphi(t) \in C^{\infty}(]0, T[)$ real-valued such that $a \equiv 1$ on $\omega \times ]0, T[$ and $\varphi \equiv 1$ 
on $\left[\frac{T}{3}, \frac{2T}{3} \right]$ there exists $C:=C(T) > 0$ such that 
\begin{eqnarray}\label{control-LS}
\|\phi_{0}\|^{2}_{L^{2}(\mathbb{T})} \leq  
C\int_{0}^{T}{\|a(x)\varphi(t)e^{it\partial_{x}^{2}}\phi_{0}\|_{L^{2}(\mathbb{T})}^{2} }dt.
\end{eqnarray}
\end{proposition}

Additionally, an observability inequality for the single KdV equation can be verified and we cite \cite{Russel93} for more details. 

\begin{proposition}
Let $T> 0$ be given. There exists $C = C(T) > 0$ such that, 
\begin{eqnarray}\label{control-KdV}
\|\psi_{0}\|_{L^{2}_{0}(\mathbb{T})}^{2} 
\leq C\int_{0}^{T}{\|Ge^{-t(\partial_{x}^{3} + \mu \partial_{x})}\psi_{0}\|_{L^{2}_{0}(\mathbb{T})}^{2}}dt.
\end{eqnarray}
\end{proposition}

As a consequence of the previous observability inequality, using the HUM method introduced by Lions \cite{Lions}, the exact controllability in $L^2(\mathbb{T})\times L^2_0(\mathbb{T})$ for the linear system \eqref{prob-contr_lin} holds. Indeed, just observe that combining both observability estimates \eqref{control-LS} and \eqref{control-KdV}, we get
$$
\|(\phi_{0}, \psi_{0})\|_{L^{2}(\mathbb{T}) \times L^{2}_{0}(\mathbb{T})}^{2} \leq C\left( \int_{0}^{T}{\|a(\cdot)\varphi(t)u(t) \|_{L^{2}(\mathbb{T})}^{2}}dt + \int_{0}^{T}{\|Ge^{-t(\partial_{x}^{3} + \mu \partial_{x})}\psi(t)\|_{L^{2}_{0}(\mathbb{T})}^{2}}dt\right),
$$
which means that, the linear system \eqref{prob-contr_lin} is null controllable, that is, the map
\begin{eqnarray*}
S: L^{2}(\mathbb{T}) \times L^{2}_{0}(\mathbb{T}) &\to& L^{2}(\mathbb{T}) \times L^{2}_{0}(\mathbb{T})\\ 
S(\varphi_{0}, \psi_{0}) &=& (u(0), v(0)),
\end{eqnarray*}
where $(u(0), v(0))$ is the initial data associated to 
\eqref{prob-contr_lin} with $f(x,t)$ and $h(x,t)$ defined by \eqref{fg}, is an isomorphism and the following linear result is verified. 

\begin{theorem}\label{local}
Let $\omega \subset \mathbb{T}$ be a nonempty open set and $T> 0$. Then  for every 
$(u_{0}, v_{0}), (u_{1}, v_{1}) \in L^{2}(\mathbb{T}) \times L^{2}_{0}(\mathbb{T})$ 
one can find two control inputs $f \in C([0, T]; L^{2}(\mathbb{T}))$ and $h\in C([0, T]; L^{2}_{0}(\mathbb{T}))$, such that, the unique solution $(u, v) \in C([0, T]; L^{2}(\mathbb{T}) \times L^{2}_{0}(\mathbb{T}))$ of \eqref{prob-contr_lin} satisfies $(u,v)(x,T) = (u_{1}(x), v_{1}(x))$.
\end{theorem}

Now, we are in a position to prove the Theorem \ref{exc_local}.

 \begin{proof}[Proof of Theorem \ref{exc_local}]
For $T> 0$ to be determined, we consider the systems 
\begin{equation*}
\left\{\begin{array}{ll}
i\partial_{t}\phi + \partial_{x}^{2}\phi = 0,  & (x,t)\in \mathbb{T} \times (0,T),\\
\partial_{t}\psi + \partial_{x}^{3}\psi + \mu\partial_{x}\psi = 0,  & (x,t)\in \mathbb{T} \times (0,T),\\
\phi(x,0) = \phi_{0}(x),\  \psi(x,0) = \psi_{0}(x),  & x\in \mathbb{T}
\end{array}\right.
\end{equation*}
 and 
\begin{equation*}
\left\{\begin{array}{ll}
 i\partial_{t}u + \partial_{x}^{2}u = i\partial_{x}v + \beta|u|^{2}u + f ,  & (x,t)\in \mathbb{T} \times (0,T),\\
\partial_{t}v + \partial_{x}^{3}v + \mu\partial_{x}v + \frac{1}{2}\partial_{x}(v^{2}) = \operatorname{Re}(\partial_{x}u) + Gh,  & (x,t)\in \mathbb{T} \times (0,T),\\
(u(\cdot, T), v(\cdot, T)) = (0, 0),  & x\in \mathbb{T},
\end{array}\right.
\end{equation*}
with $f:=a^{2}\varphi^{2}\phi$ and $h:=G^{*}\psi=G\psi$.     Let us define the operator 
 \begin{eqnarray*}
 L: L^{2}(\mathbb{T})\times L^{2}_{0}(\mathbb{T})) &\to& L^{2}(\mathbb{T})\times L^{2}_{0}(\mathbb{T}) \\
 (\phi_{0}, \psi_{0}) &\mapsto& L(\phi_{0}, \psi_{0}) = (u|_{t = 0}, v|_{t = 0}) = (u_{0}, v_{0})
 \end{eqnarray*}
 We split $(u,v)$ into
 $$(u, v) = (u_{L}, v_{L}) + (u_{NL}, v_{NL}) $$
where $(u_{L}, v_{L})$ is solution of 
\begin{equation*}
\left\{\begin{array}{ll}
 i\partial_{t}u_{L} + \partial_{x}^{2}u_{L} = a^{2}\varphi^{2}\phi,  & (x,t)\in \mathbb{T} \times (0,T),\\
\partial_{t}v_{L} + \partial_{x}^{3}v_{L} + \mu \partial_{x}v_{L} = GG^{*}\psi,  & (x,t)\in \mathbb{T} \times (0,T),\\
(u_{L}, v_{L})(\cdot,T) = (0, 0),& x\in \mathbb{T}.
\end{array}\right.
\end{equation*}
and $(u_{NL}, v_{NL})$ is solution of 
\begin{equation*}
\left\{\begin{array}{ll}
 i\partial_{t}u_{NL} + \partial_{x}^{2}u_{NL} = i\partial_{x}v + \beta|u|^{2}u,  & (x,t)\in \mathbb{T} \times (0,T),\\
 \partial_{t}v_{NL} + \partial_{x}^{3}v_{NL} + \mu \partial_{x}v_{NL} = \operatorname{Re}(\partial_{x}u) - \frac{1}{2}\partial_{x}(v^{2}),  & (x,t)\in \mathbb{T} \times (0,T),\\
(u_{NL}, v_{NL})(\cdot,T) = (0,0),& x\in \mathbb{T}.
\end{array}\right.
\end{equation*}

Observe that we can follow \cite{dehman-gerard-lebeau} to construct an isomorphism  $$S: L^2(\mathbb{T})\times L^{2}_{0}(\mathbb{T})\to L^2(\mathbb{T})\times L^{2}_{0}(\mathbb{T})$$ such that $(u_{L}(0), v_{L}(0)) = S_{L}(\phi_{0}, \psi_{0}).$ Additionally, we can also construct another application  (see e.g. \cite[Lemma 2.4]{Laurent}) $$K: L^2(\mathbb{T})\times L^{2}_{0}(\mathbb{T})\to L^2(\mathbb{T})\times L^{2}_{0}(\mathbb{T})$$ which satisfies $K(\phi_{0}, \psi_{0}) = (u_{NL}(0), v_{NL}(0))$. With these information in hand, we have that if $(u, v), (u_{L}, v_{L}), (u_{NL}, v_{NL}) \in \widetilde{X}_{T}^{0} \times \widetilde{Y}_{T}^{0}$ and 
\begin{eqnarray*}
(u,v)(0) = (u_{L}, v_{L})(0) + (u_{NL}, v_{NL})(0),
\end{eqnarray*}
then we can rewrite the previous equality as follows 
\begin{eqnarray*}
L(\phi_{0}, \psi_{0}) = S_{L}(\phi_{0}, \psi_{0}) + K(\phi_{0}, \psi_{0}),
\end{eqnarray*}
where $K(\phi_{0}, \psi_{0}) = (u_{NL}(0), v_{NL}(0))$.  So, we have that $L(\phi_{0}, \psi_{0}) = (u_{0}, v_{0})$ is equivalent to 
\begin{eqnarray*}
(\phi_{0}, \psi_{0}) = S_{L}^{-1}(u_{0}, v_{0}) - S_{L}^{-1}K(\phi_{0}, \psi_{0}).
\end{eqnarray*}
Thus, let us define the following map
$$
B: L^{2}(\mathbb{T}) \times L^{2}_{0}(\mathbb{T}) \to L^{2}(\mathbb{T}) \times L^{2}_{0}(\mathbb{T})$$
by 
$$B(\phi_{0}, \psi_{0}) = S_{L}^{-1}(u_{0}, v_{0}) - S_{L}^{-1}K(\phi_{0}, \psi_{0}). $$
Therefore, our null controllability problem is reduced to prove that $B$ has a fixed point, so let us prove it now.  

From now on we may fix $T < 1$. Since $S_{L}$ is an isomorphism of 
$L^{2}(\mathbb{T}) \times L^{2}_{0}(\mathbb{T})$, we have,
\begin{eqnarray*}
\|B(\phi_{0}, \psi_{0}) \|_{L^{2}(\mathbb{T}) \times L^{2}_{0}(\mathbb{T})} \leq C\left(\|K(\phi_{0}, \psi_{0})\|_{L^{2}(\mathbb{T}) \times L^{2}_{0}(\mathbb{T})}  + \|(u_{0}, v_{0})\|_{L^{2}(\mathbb{T}) \times L^{2}_{0}(\mathbb{T})} \right).
\end{eqnarray*}
We are interested to estimate $\|K(\phi_{0}, \psi_{0})\|_{L^{2}(\mathbb{T}) \times L^{2}_{0}(\mathbb{T})} = \|(u_{NL}(0), v_{NL}(0))\|_{L^{2}(\mathbb{T}) \times L^{2}_{0}(\mathbb{T})}$. To make it we use the Duhamel formula and the  $X_{T}^{0, \frac{1}{2}} \times Y_{T}^{0, \frac{1}{2}}$--estimates, to conclude that
\begin{equation*}
\begin{split}
\|(u_{NL}(0), v_{NL}(0))\|_{L^{2}(\mathbb{T}) \times L^{2}_{0}(\mathbb{T})} \leq& \|(u_{NL}, v_{NL})\|_{\widetilde{X}_{T}^{0} \times \widetilde{Y}_{T}^{0}} \\
 \leq &\ C\| |u|^{2}u\|_{X_{T}^{0, - \frac{1}{2}}} + C\|\partial_{x}v\|_{X_{T}^{0, -\frac{1}{2}}} \\&+ C\|\operatorname{Re}(\partial_{x}u)\|_{{Y}_{T}^{0, - \frac{1}{2}}} + C \|\partial_{x}(v^{2})\|_{Y_{T}^{0, - \frac{1}{2}}} \\ 
\leq&\ CT^{0+}\|u\|_{X_{T}^{0, \frac{1}{2}}}^{3} + CT^{0+}\|u\|_{X_{T}^{0, \frac{1}{2}}} + CT^{0+}\|v\|_{Y^{0, \frac{1}{2}}_{T}} + CT^{0+}\|v\|_{Y_{T}^{0+}}^{2},
\end{split}
\end{equation*}
thanks to the estimates of Section \ref{Sec2}.
Now, we notice that, 
\begin{equation*}
\|(u, v)\|_{X_{T}^{0, \frac{1}{2}}} \leq C\|(\varphi^{2}a^{2}\phi, GG^{*}\psi)\|_{L^{2}(0, T; L^{2}(\mathbb{T})\times L^{2}_0(\mathbb{T}))} 
\leq\ C\|(\phi_{0}, \psi_{0})\|_{L^{2}(\mathbb{T})\times L^{2}_0(\mathbb{T})} < C\eta 
\end{equation*}
and then, 
$$
\|B(\phi_{0}, \psi_{0})\|_{L^{2}(\mathbb{T}) \times L^{2}_{0}(\mathbb{T})} \leq C\left(T^{0+}\eta + T^{0+}\eta^{2} + \|(u_{0}, v_{0})\|_{L^{2}(\mathbb{T}) \times L^{2}_{0}(\mathbb{T})} \right).
$$
Therefore, taking $\|(u_{0}, v_{0})\|_{L^{2}(\mathbb{T}) \times L^{2}_{0}(\mathbb{T})}$ small enough, we can conclude that $B$ maps the closed ball $B_{\eta}$ of $L^{2}(\mathbb{T})\times L^{2}(\mathbb{T})$ into itself. 
 
It remains to prove that $B$ is a contraction. We start by considering  the following two systems 
\begin{equation*}
\left\{\begin{array}{ll}
i\partial_{t}(u - \widetilde{u}) + \partial_{x}^{2}(u - \widetilde{u}) = i\partial_{x}(v - \widetilde{v}) + \beta\left(|u|^{2}u - |\widetilde{u}|^{2}\widetilde{u} \right) + 
a^{2}\varphi^{2}(\phi - \widetilde{\phi}) \\
\partial_{t}(v - \widetilde{v}) + \partial_{x}^{3}(v - \widetilde{v}) + \mu\partial_{x}(v - \widetilde{v}) + \frac{1}{2}\partial_{x}(v^{2} - \widetilde{v}^{2}) = \operatorname{Re}(\partial_{x}u - \partial_{x}\widetilde{u}) + GG^{*}(\psi - \widetilde{\psi}) \\
(u - \widetilde{u})(\cdot,T) = 0, \quad (v - \widetilde{v})(\cdot,T) = 0.
\end{array}\right.
\end{equation*}
and 
\begin{equation*}
\left\{\begin{array}{ll}
 i\partial_{t}(u_{NL} - \widetilde{u}_{NL}) 
+ \partial_{x}^{2}(u_{NL} - \widetilde{u}_{NL}) = \beta\left(|u|^{2}u - |\widetilde{u}|^{2}\widetilde{u} \right) + i\partial_{x}(v - \widetilde{v})\\
 \partial_{t}(v_{NL} - \widetilde{v}_{NL}) + \partial_{x}^{3}(v_{NL} - \widetilde{v}_{NL}) + \mu \partial_{x}(v_{NL} - \widetilde{v}_{NL}) = 
\operatorname{Re}(\partial_{x}u - \partial_{x}\widetilde{u}) + \frac{1}{2}\left(\partial_{x}(\widetilde{v}^{2} - v^{2}) \right)\\
 (u_{NL} - \widetilde{u}_{NL})(\cdot,T) = 0, \quad (v_{NL} - \widetilde{v}_{NL})(\cdot,T) = 0.
\end{array}\right.
\end{equation*}
Again, due to the estimates provided in Section \ref{Sec2}, we get 
\begin{equation*}
\begin{split}
\| B(\phi_{0}, \psi_{0}) - B(\widetilde{\phi_{0}}, \widetilde{\psi_{0}})\|_{L^{2}(\mathbb{T}) \times L^{2}_{0}(\mathbb{T})} \leq&\ \|((u_{NL} - \widetilde{u}_{NL})(0), (v_{NL} - \widetilde{v}_{NL})(0))\|_{L^{2}(\mathbb{T})\times L^{2}_{0}(\mathbb{T})}\\
\leq&\ C\| |u|^{2}u - |\widetilde{u}|^{2}\widetilde{u}\|_{X_{T}^{0, - \frac{1}{2}}} + C\|\partial_{x}v - \partial_{x}\widetilde{v}\|_{X_{T}^{0, - \frac{1}{2}}} \\&\ + C\|\partial_{x}(v^{2}) - \partial_{x}(\widetilde{v}^{2})\|_{Y_{T}^{0, - \frac{1}{2}}} + C\|\operatorname{Re}(\partial_{x}u - \partial_{x}\widetilde{u})\|_{Y_{T}^{0, - \frac{1}{2}}}  \\
\leq&\ CT^{0+}\left(1 + \|u\|_{X_{T}^{0, \frac{1}{2}}}^{2} + \|\widetilde{u}\|_{X_{T}^{0, \frac{1}{2}}}^{2} \right)\|u - \widetilde{u}\|_{X_{T}^{0,\frac{1}{2}}} \\
&\ + CT^{0+}\left(1 + \|v\|_{Y_{T}^{0, \frac{1}{2}}} + \|\widetilde{v}\|_{Y_{T}^{0, \frac{1}{2}}} \right)\|v - \widetilde{v}\|_{Y_{T}^{0, \frac{1}{2}}} \\
\leq& \ C\left(T^{0+}\eta^{2} + T^{0+}\right)\|u - \widetilde{u}\|_{X_{T}^{0, \frac{1}{2}}} \\
&\ + \left( T^{0+} + T^{0+}\eta\right)\|v - \widetilde{v}\|_{Y_{T}^{0, \frac{1}{2}}}.
\end{split}
\end{equation*}
Considering the system \eqref{prob-contr}, we deduce that 
\begin{equation}\label{r1}
\begin{split}
\|u - \widetilde{u}\|_{X_{T}^{0, \frac{1}{2}}} + \|v - \widetilde{v}\|_{Y_{T}^{0, \frac{1}{2}}}\leq&\ C\| |u|^{2}u - |\widetilde{u}|^{2}\widetilde{u}\|_{X_{T}^{0, \frac{1}{2}}} + C\|\partial_{x}v - \partial_{x}\widetilde{v}\|_{X_{T}^{0, \frac{1}{2}}} \\&
+ C\|\partial_{x}(v^{2}) - \partial_{x}(\widetilde{v}^{2}) \|_{Y_{T}^{0, \frac{1}{2}}} + \|\operatorname{Re}(\partial_{x}u) - \operatorname{Re}(\partial_{x}\widetilde{u})\|_{Y_{T}^{0, \frac{1}{2}}} \\ & + C\|\varphi^{2}a^{2}(\phi - \widetilde{\phi})\|_{L^{2}(\mathbb{T})} + \|GG^{*}(\psi - \widetilde{\psi})\|_{L^{2}_{0}(\mathbb{T})} \\
\leq& \ C(T^{0+}\eta^{2} + T^{0+})\|u - \widetilde{u}\|_{X_{T}^{0, \frac{1}{2}}} + 
(T^{0+} + T^{0+}\eta^{2})\|v - \widetilde{v}\|_{Y_{T}^{0, \frac{1}{2}}} \\&+ C\|\phi_{0} - 
\widetilde{\phi}_{0}\|_{L^{2}(\mathbb{T})} + C\|\psi_{0} - \widetilde{\psi}_{0}\|_{L^{2}_{0}(\mathbb{T})}.
\end{split}
\end{equation}
Then, for $\eta < 1$ fixed, we can get a small $T > 0$ such that,
\begin{equation}\label{r2}
\|u - \widetilde{u}\|_{X_{T}^{0, \frac{1}{2}}} + \|v - \widetilde{v}\|_{Y_{T}^{0, \frac{1}{2}}}\leq C\|\phi_{0} - 
\widetilde{\phi}_{0}\|_{L^{2}(\mathbb{T})} + C\|\psi_{0} - \widetilde{\psi}_{0}\|_{L^{2}_{0}(\mathbb{T})}.
\end{equation}
Thanks to the inequalities \eqref{r1} and \eqref{r2}, we get that
$$
\|B(\phi_{0}, \psi_{0}) - B(\widetilde{\phi}_{0}, \widetilde{\psi}_{0})\|_{L^{2}(\mathbb{T}) \times L^{2}_{0}(\mathbb{T})} 
\leq C\left(T^{0+} + \eta T^{0+} \right)\|(\phi_{0}, \psi_{0}) - (\widetilde{\phi}_{0}, \widetilde{\psi}_{0})\|_{L^{2} (\mathbb{T})\times L^{2}_{0}(\mathbb{T})}.
$$
To finish, just choosing  $$T^{0+} + \eta T^{0+}<\frac{1}{2C},$$ we conclude that $B$ is a contraction in the closed ball of $L^{2}(\mathbb{T}) \times L^{2}_{0}(\mathbb{T})$ and the local controllability result is verified and Theorem \ref{exc_local} is achieved. 
 \end{proof}

\section{Further comments and open issues}\label{Sec6}
In this work, we showed the stabilization and control issues of the system  
\begin{equation}\label{prob-contr-intra}
\left\{\begin{array}{ll}
i \partial_{t} u+\partial_{x}^{2} u = i \partial_{x} v+\beta|u|^{2} u + f, & (x,t)\in\mathbb{T} \times (0,T),\\
\partial_{t} v+\partial_{x}^{3} v + \frac{1}{2} \partial_{x}\left(v^{2}\right) + \mu\partial_{x}v = \operatorname{Re}\left(\partial_{x} u\right) + g,  & (x,t)\in \mathbb{T} \times (0,T),\\
u(x, 0)=u_{0}(x), \quad v(x, 0)=v_{0}(x), & x \in \mathbb{T}.
\end{array}\right.
\end{equation} 
When we introduce two damping mechanisms $f:=-ia^2(x)u$ and $g:=-GG^*h$,  where $a$ satisfies \eqref{Def-a} and $G$ is defined by \eqref{Def-G}, we can prove that the solutions associated to \eqref{prob-contr-intra} decay exponentially for $t$ large enough. Moreover, together with local exact controllability, we are also able to prove a global exact controllability property, that means, for any initial and final data, without restrictions on its norms, we can find two control inputs $ f $ and $ g $ such that we can drive the initial data to the final data by using these controls. 

The strategy employed in the proof of the main results in this manuscript uses the contradiction argument, which aims to prove an observability inequality. Observability inequality follows by a combination of propagation of compactness and regularity, together with a unique continuation property for the operator associated with the NLS--KdV system, which is proved using the defect measure.

\subsection{Only one control acting} Control problems of coupled dispersive systems are still not well understood and most of the results consider single equations. In our case, we are able to combine two dispersive equations with nonlinearities and principal symbols of a different order. One of the main difficulties in our case arises in the proof of the propagation results,  and was overcome using the smoothing properties of the solutions in Bourgain space.

This method presented here seems to work well however it has a drawback. When we decided to study the nonlinear problem directly, instead of first the linear problem, we were unable to remove controls, that is, with this method, the best possible result is to show the control results with a control in each equation. 

It is important to mention that in \cite{Arauna}, the authors proved that the following linear Schr\"odinger--KdV system
\begin{equation}\label{oneCon}
\left\{\begin{array}{lll}
i \partial_tw+\partial^2_xw=a_{1} w+a_{2} y+h \mathbf{1}_{\omega} & \text { in } & Q \\
\partial_ty+\partial^3_xy+\partial_x(M y)=\operatorname{Re}\left(a_{3} w\right)+a_{4} y  + l \mathbf{1}_{\omega}& \text { in } & Q \\
w(0, t)=w(1, t)=0 & \text { in } & (0, T), \\
y(0, t)=y(1, t)=\partial_xy(1, t)=0 & \text { in } & (0, T), \\
w(x, 0)=w_{0}(x), \quad y(x, 0)=y_{0}(x) & \text { in } & (0,1),
\end{array}\right.
\end{equation}
in a bounded domain $Q:=(0,1)\times(0,T)$, is controllable.  They are not able to produce any nonlinear results due to the lack of regularity to treat the nonlinear problem in Sobolev spaces,  in this way, our work extends, for the nonlinear system, the control problems for a more general system. The main point of dealing with nonlinear structure in this manuscript is to work in the Bourgain spaces, which helps us to treat the nonlinearities appropriately.  However,  we mention that the linear case helps to deal with the system \eqref{oneCon} with only a control $h\mathbf{1}_{\omega}$.  This was possible by using a Carleman estimate which involves both operators, Schrödinger, and KdV operators. However,  it is important to make it clear that in our case we can not use this inequality to remove one control, since the regularity used in \cite[Theorem 1.3.]{Arauna} is not enough to deal with the nonlinearities involved in the system \eqref{prob-contr-intra}. 

\subsection{Others nonlinearities} Another interesting problem is to consider the full physical model of the NLS--KdV system that appears in \cite{Arbieto},
\begin{equation}\label{Arb}
\left\{\begin{array}{ll}
i \partial_{t} u+\partial_{x}^{2} u=\alpha u v+\beta|u|^{2} u+ f, & (x,t)\in\mathbb{T} \times (0,T),\\
\partial_{t} v+\partial_{x}^{3} v+\frac{1}{2} \partial_{x}\left(v^{2}\right)=\gamma \partial_{x}\left(|u|^{2}\right) + g, & (x,t)\in\mathbb{T} \times (0,T),\\
u(x, 0)=u_{0}(x), \quad v(x, 0)=v_{0}(x), & x \in \mathbb{T},
\end{array}\right.
\end{equation} 
to prove controllability results. In this way, we propose the following natural issue.

\vspace{0.2cm}

\noindent\textbf{Question $\mathcal{A}$:} Is it possible to prove control results for the system \eqref{Arb}?

\vspace{0.2cm}

Note that, in this case, to reach the result we need to study some conservation laws associated with the solution of the system. In \cite[Lemma 5.1.]{Arbieto} the authors proved that the evolution system \eqref{Arb} preserves the following quantities
$$
M(t):=\int_{\mathbb{T}}|u(t)|^{2} d x, \quad \quad Q(t):=\int_{\mathbb{T}}\left\{\alpha v(t)^{2}+2 \gamma \Im\left(u(t) \overline{\partial_{x} u(t)}\right)\right\} d x$$
and 
$$
E(t):=\int_{\mathbb{T}}\left\{\alpha \gamma v(t)|u(t)|^{2}-\frac{\alpha}{6} v(t)^{3}+\frac{\beta \gamma}{2}|u(t)|^{4}+\frac{\alpha}{2}\left|\partial_{x} v(t)\right|^{2}+\gamma\left|\partial_{x} u(t)\right|^{2}\right\} d x.
$$
These quantities give us only information about the ``mass" $M(t)$ associated with the Schrödinger equation and there is no information about this conversation for the part of the KdV equation. So, it is necessary, before answering the Question $\mathcal{A}$, to solve the following problem.

\vspace{0.2cm}

\noindent\textbf{Question $\mathcal{B}$:}  \textit{How to deal with the KdV part presented in \eqref{Arb} to present some information for the ``mass"? }

\vspace{0.2cm}

This problem is an open issue and it seems that our work opens, from now on, some possibilities to deal with these types of systems. 

 \appendix
 \section{Propagation results in Bourgain spaces}\label{Appendix1}
This appendix is dedicated to presenting properties of propagation in Bourgain spaces for the linear differential operator associated with the NLS--KdV system.  These results of propagation are the key to proving the global control results, that some parts were borrowed from \cite{Laurent-esaim, Laurent}. For self-contentedness, we will also give a rigorous proof of them. The main ingredient is pseudo-differential analysis.  Let us begin with the result of the propagation of compactness which will ensure strong convergence in appropriate spaces.
 
\begin{proposition}
[\textit{Propagation of compactness}]\label{Prop-Comp}
Let $T > 0$, suppose that $(u_{n}, v_{n}) \in X^{0, \frac{1}{2}}_{T} \times Y_{T}^{0, \frac{1}{2}}$ and $(f_{n}, g_{n}) \in X_{T}^{- 1, - \frac{1}{2} } \times Y_{T}^{- 1, - \frac{1}{2}}$, $n \in \mathbb{N}$, satisfies
\begin{equation*}
\left\{\begin{array}{ll}
i \partial_{t} u_{n}+\partial_{x}^{2} u_{n}=i \partial_{x}v_{n} + f_{n},  & (x,t)\in \mathbb{T} \times (0,T),\\
\partial_{t} v_{n}+\partial_{x}^{3} v_{n} + \mu \partial_{x}v_{n}=R e\left(\partial_{x}u_{n}\right) + g_{n},  & (x,t)\in \mathbb{T} \times (0,T).
\end{array}\right.
\end{equation*}
Assume that there exists a constant $C > 0$ such that 
\begin{eqnarray} \label{bound-prop}
\|(u_{n}, v_{n})\|_{X_{T}^{0, \frac{1}{2}} \times Y_{T}^{0, \frac{1}{2}}} \leq C, \quad \text{ for } n \geq 1. 
\end{eqnarray}
Additionally, suppose that  
$$
(u_{n}, v_{n}) \to (0, 0) \quad \hbox{weakly in} \quad X^{0,  \frac{1}{2}}_{T} \times Y_{T}^{0, \frac{1}{2}},
$$
$$
(f_{n}, g_{n}) \to (0, 0) \quad \hbox{strongly in } \quad X^{-1, - \frac{1}{2}}_{T} \times Y_{T}^{-1 , -\frac{1}{2}}.
$$
and that for some nonempty set $\omega \subset \mathbb{T}$ we have 
\begin{eqnarray*}
(u_{n}, v_{n}) \to (0, 0) \text{ in } L^{2}(0, T; L^{2}(\omega) \times L^{2}_0(\omega)).
\end{eqnarray*}
Then, 
\begin{eqnarray}
(u_{n}, v_{n}) \to (0,0) \text{ strongly in } L^{2}_{loc}(0, T; L^{2}(\mathbb{T}) \times L^{2}_{0}(\mathbb{T})). 
\end{eqnarray}
\end{proposition}
 
 \begin{proof}
We shall start by proving the following convergence 
\begin{eqnarray}
u_{n} \to 0 \text{ strongly in } L^{2}_{loc}(0, T; L^{2}(\mathbb{T})).
\end{eqnarray}
Denote $L_{1} := i\partial_{t} + \partial_{x}^{2}$ the linear 
Schrödinger operator and consider $\varphi \in C^{\infty}(\mathbb{T})$ and 
$\psi \in C_{0}^{\infty}(]0, T[)$ taking real values to be determined. Define $Bu := \varphi(x)D^{-1}$ and $A = \psi(t)B$ 
and for $\epsilon > 0$, $A_{\epsilon} = Ae^{\epsilon \partial_{x}^{2}} = \psi(t)B_{\epsilon}$, with $D^{-1}$ defined by \eqref{DR}. We consider,
\begin{equation*}
\alpha_{n,\epsilon} := \left([A_{\epsilon}, L_{1}]u_{n}, u_{n} \right)_{L^{2}(0, T; L^{2}(\mathbb{T}))}=  \left([A_{\epsilon}, \partial_{x}^{2}]u_{n}, u_{n} \right)- i\left(\psi'(t)Bu_{n}, u_{n}\right). 
\end{equation*}
On the other hand, we also have
\begin{equation}\label{prop-comp}
\alpha_{n, \epsilon} = \ (i\partial_{x}v_{n}, A_{\epsilon}^{*}u_{n}) - (A_{\epsilon}u_{n}, i\partial_{x}v_{n}) +  (f_{n}, A_{\epsilon}^{*}u_{n}) - (A_{\epsilon}u_{n}, f_{n}) .
\end{equation}

Now, observe that using Lemma \ref{LemaBou} we have
\begin{equation}\label{prop-comp1}
\begin{split}
|(i\partial_{x}v_{n}, A_{\epsilon}^{*}u_{n})_{L^{2}_{t,x}}| \leq & \sum_{m \in \mathbb{Z}}\int_{-\infty}^{\infty}{ {\left|\widehat{\partial_{x}v_{n}(m, \tau)} \right| \left|\overline{\widehat{A_{\epsilon}^{*}u_{n}}(n, \tau) } \right|} }d\tau \\ 
 \leq & \sum_{m \in \mathbb{Z}}\int_{-\infty}^{\infty}{{H(\tau, n)\langle \tau -n^{3} + \mu n \rangle^{b'} \langle n \rangle^{-1}|\widehat{v_{n}}(n, \tau)| \langle \tau + n^{2} \rangle^{b'}\langle n \rangle^{1}|\overline{\widehat{A^{*}_{\epsilon}u_{n}}(n, \tau)}|}}d\tau \\ 
 \leq &\ C\left(\sum_{m \in \mathbb{Z}}\int_{-\infty}^{\infty}{\langle \tau - n^{3}+\mu n \rangle^{2b'}\langle n\rangle^{-2} \left|\widehat{v}_{n}(m, \tau)\right|^{2}}d\tau \right)^{\frac{1}{2}}\times \\& \left(\sum_{m \in \mathbb{Z}}\int_{-\infty}^{\infty}{\langle \tau + n^{2}\rangle^{2b'}  \langle n\rangle^{2}\left|\overline{\widehat{A_{\epsilon}^{*}u_{n}}(n, \tau)}\right|^{2}}d\tau \right)^{\frac{1}{2}}  \\ 
 \leq &\ C\|v_{n}\|_{Y_{T}^{-1, b'}} \|A_{\epsilon}^{*}u_{n}\|_{X^{1, b'}_{T}} \\  
 \leq &\ C \|v_{n}\|_{Y_{T}^{-1, b'}} \|u_{n}\|_{X^{0, b'}_{T}}.
\end{split}
\end{equation}
Noting that, for $ b' = \frac{1}{2}-$, we have $$v_{n} \to 0
\quad \text{strongly in } Y_{T}^{-1, b'}$$ and  thanks to \eqref{bound-prop} we get that $\|u_{n}\|_{X^{0, b'}_{T}}$ is bounded. Since $$v_{n} \to 0\quad \text{in }Y^{-1, b'}_{T},$$ yields that 
 $$\sup_{0 < \epsilon \leq 1}|(i\partial_{x}v_{n}, A_{\epsilon}^{*}u_{n})_{L^{2}_{t,x}}| \to 0,$$ 
 due to \eqref{prop-comp1}. Additionally,  holds that
$$
|\left(f_{n}, A_{\epsilon}^{*}u_{n} \right)| \leq  \|f_{n}\|_{X_{T}^{-1, -\frac{1}{2}}} \|u_{n}\|_{X_{T}^{0, \frac{1}{2}}},
$$
which yields that $$\sup_{1<\epsilon\leq0}|\left(f_{n}, A_{\epsilon}^{*}u_{n} \right)| \to 0, \quad \text{as } n\to\infty.$$ The same kind of estimates for the other terms in \eqref{prop-comp} give us the following limit
\begin{eqnarray*}
\sup_{1<\epsilon\leq0}|\alpha_{n, \epsilon}| \to 0, \quad\text{when }n\to\infty.
\end{eqnarray*}
Noting that  
$$[A, \partial_{x}^{2}] = -2\psi(t)(\partial_{x} \varphi)\partial_{x}D^{-1} - \psi(t)(\partial_{x}^{2})D^{-1},
$$
since $D^{-1}$ commutes with $\partial_{x}$, and taking the supremum on $\epsilon$ tending to $0$, we conclude that 
\begin{eqnarray*}
\left([A, \partial_{x}^{2}]u_{n}, u_{n} \right) \to 0,\quad \text{ as } n \to \infty.
\end{eqnarray*}
Putting together these convergences, above mentioned, we may have that
\begin{eqnarray*}
(\psi(t)(\partial_{x}^{2}\varphi)D^{-1}u_{n}, u_{n}) \to 0\quad \text{ as } n \to \infty.
\end{eqnarray*}
Note that, $-i\partial_{x}D^{-1}$ is the orthogonal projection 
on the subspace of functions with $\widehat{u}(0) = 0$, and 
$\widehat{u}_{n}(0)(t)$ tends to $0$ in $L^{2}([0, T])$ we have 
\begin{eqnarray*}
(\psi(t)(\partial_{x}\varphi)\widehat{u_{n}}(0)(t), u_{n}) \to 0, \quad \text{ as } n \to \infty.
\end{eqnarray*} 
Thus we have proved that for any $\varphi \in C^{\infty}(\mathbb{T})$ and $\psi \in C^{\infty}_{0}(]0, T[)$, 
\begin{eqnarray*}
(\psi(t)(\partial_{x}\varphi)u_{n}, u_{n}) \to 0, \quad \text{ as } n \to \infty.
\end{eqnarray*} 
We conclude the first part of this proof by observing that 
a function $\phi \in C^{\infty}(\mathbb{T})$ can be written as 
$\partial_{x}\varphi$ for some $\varphi \in C^{\infty}(\mathbb{T})$ if and only if $\int_{\mathbb{T}}{\varphi}dx = 0$. 
Thus for any $\chi \in C^{\infty}_{0}(\omega)$ and any 
$x_{0} \in \mathbb{T}$, $\phi(x) = \chi(x) - \chi(\cdot - x_{0})$ 
can be written as $\phi = \partial_{x}\varphi$ for some 
$\varphi \in C^{\infty}(\mathbb{T})$. 

Now we use the hypothesis that $u_{n} \to 0$ strongly in 
$L^{2}(0, T; L^{2}(\omega))$ to conclude that 
\begin{eqnarray*}
\lim_{n \to \infty}(\psi(t)\chi u_{n}, u_{n}) = 0.
\end{eqnarray*}
Therefore, for any $x_{0} \in \mathbb{T}$,
\begin{eqnarray*}
\lim_{n \to \infty}(\psi(t)\chi(\cdot - x_{0})u_{n}, u_{n}) = 0
\end{eqnarray*}
and yields that $u_{n} \to 0$ in $L^{2}_{loc}(0, T; L^{2}(\mathbb{T}))$ by constructing a partition of the unity 
of $\mathbb{T}$ involving functions of the form 
$\chi_{i}(\cdot - x_{0}^{i})$ with $\chi_{i} \in C^{\infty}(\omega)$ and $x_{0}^{i} \in \mathbb{T}$.

To prove that $v_{n} \to 0$ strongly in $L^{2}_{loc}(0, T; L^{2}_0(\mathbb{T}))$ we argue analogously, but in this case considering $L_{2} = \partial_{t}+ \partial_{x}^{3} + \mu \partial_{x}$, $B_{1} = \varphi(x)D^{-2}$ and $A_{2} = \psi(t)B_{2}$ and for $\epsilon > 0$, $A_{2\epsilon} = A_{2}e^{\epsilon \partial_{x}^{2}}$. In this situation, we deal with 
\begin{eqnarray*}
\widetilde{\alpha_{n, \epsilon}} = \left([A_{2\epsilon}, L_{2}]v_{n}, v_{n} \right)_{L^{2}_{t,x}}
\end{eqnarray*}
the novelty, in this case, is that
\begin{eqnarray*}
\widetilde{\alpha_{n, \epsilon}} &=& (g_{n}, A_{2\epsilon}^{*}v_{n}) + (A_{2\epsilon}v_{n}, g_{n}) + (Re(\partial_{x}u_{n}), v_{n}) +
(v_{n}, Re(\partial_{x}u_{n}))
\end{eqnarray*}
and the terms $(Re(\partial_{x}u_{n}), v_{n}) +
(v_{n}, Re(\partial_{x}u_{n}))$ can be dealt as in \eqref{prop-comp1}. The rest of the proof is then completed following the same steps as the first case and the proof is complete.
\end{proof}

We now prove the propagation of regularity for the linear operator associated with the NLS--KdV system.

\begin{proposition}
[\textit{Propagation of regularity}] \label{prop-reg}
Let $T > 0$, $r \in \mathbb{R}_{+}$ and $(u,v) \in X_{T}^{r, \frac{1}{2}} \times Y_{T}^{r, \frac{1}{2}}$ solution of 
\begin{equation*}
\left\{\begin{array}{ll}
i \partial_{t} u+\partial_{x}^{2} u=i \partial_{x}v+ f,  & (x,t)\in \mathbb{T} \times (0,T),\\
\partial_{t} v + \partial_{x}^{3} v + \mu \partial_{x}v= Re\left(\partial_{x}u \right)+ g,  & (x,t)\in \mathbb{T} \times (0,T),
\end{array}\right.
\end{equation*}
with $(f, g) \in X_{T}^{r, -\frac{1}{2}} \times Y_{T}^{r, -\frac{1}{2}}$. Assume that there exists a non-empty 
open set $\omega \subset \mathbb{T}$ such that 
$$(u, v) \in L^{2}_{loc}(]0, T[; H^{r + \rho}(\omega) \times H_0^{r + \rho}(\omega)),$$ for some $0 < \rho <\frac{1}{4}$.  Then $$(u, v) \in L^{2}_{loc}(]0, T[; H^{r + \rho}(\mathbb{T}) \times H_0^{r +  \rho}(\mathbb{T})).$$
\end{proposition}

\begin{proof}
We first regularize $(u, v)$ by introducing, for each  $n \in \mathbb{N}$, 
$$
u_{n} := e^{\frac{1}{n}\partial_{x}^{2}}u, \quad v_{n} := e^{\frac{1}{n}\partial_{x}^{2}}v, 
$$
and
$$
f_{n} := e^{\frac{1}{n}\partial_{x}^{2}}f, \quad g_{n} := e^{\frac{1}{n}\partial_{x}^{2}}g.
$$
In this case, by hypothesis over $u$ and $v$, we have  $$\|(u_{n}, v_{n})\|_{X^{r, \frac{1}{2}}_{T} \times Y^{r, \frac{1}{2}}_{T}} \leq C \quad \text{and} \quad \|(f_{n}, g_{n}) \|_{X_{T}^{r, -\frac{1}{2}} \times Y_{T}^{r, -\frac{1}{2}}} \leq C,$$ for some constant $C > 0$. 

Pick $s = r + \rho$, $\varphi \in C^{\infty}(\mathbb{T})$ and 
$\psi \in C^{\infty}(]0, T[)$ taking real values.  As in the proof of Proposition $\ref{Prop-Comp}$ we define $B = D^{2s- 1} \varphi(x)$ and $A = \psi(t)B$, with $D$ defined in \eqref{DR}. Set $L_{1} = i \partial_{t} + \partial_{x}^{2}$ the Schrödinger operator and 
\begin{equation*}
\begin{split}
\alpha_{n} :=& \left(L_{1}u_{n}, A^{*}u_{n} \right)_{L^{2}_{t,x}} - \left(Au_{n}, L_{1}u_{n}\right)_{L^{2}_{t,x}} \\
 = & \left([A, \partial_{x}^{2}]u_{n}, u_{n}\right)_{L^{2}_{t,x}} - i (\psi'(t)Bu_{n}, u_{n})_{L^{2}_{t,x}} \\
 = & i(\partial_{x}v_{n}, A^{*}u_{n})_{L^{2}_{t,x}} + (f_{n}, A^{*}u_{n})_{L^{2}_{t,x}} - (Au_{n}, i\partial_{x}v_{n})_{L^{2}_{t,x}} - 
(Au_{n}, f_{n})_{L^{2}_{t,x}}.
\end{split}
\end{equation*}
Notice that, as we choose  $\rho > 0$ small 
enough, we have $r + 2\rho - \frac{1}{2} \leq r$. Hence, 
we obtain
\begin{eqnarray*}
\left| (Au_{n}, f_{n})_{L^{2}_{t,x}}\right| \leq \|Au_{n}\|_{X_{T}^{-r, \frac{1}{2}}} \|f_{n}\|_{X_{T}^{r, -\frac{1}{2}}} \leq C,\forall n \in \mathbb{N}.
\end{eqnarray*}
Arguing as in \eqref{prop-comp} we can similarly get that 
\begin{eqnarray*}
\left| (\partial_{x}v_{n}, Au_{n})_{L^{2}_{t,x}}\right| \leq C, \forall n \in \mathbb{N}
\end{eqnarray*}
and also estimate all the other terms to get that $([A, \partial_{x}^{2}]u_{n}, u_{n})$ is uniformly bounded in $n \in \mathbb{N}$.

On the other hand, notice that 
\begin{eqnarray*}
[A, \partial_{x}^{2}] = - 2 \Psi(t)D^{2s- 1}(\partial_{x}\varphi)\partial_{x} - \Psi(t)D^{2s-1}(\partial_{x}^{2}\varphi)
\end{eqnarray*}
and 
\begin{eqnarray*}
|\left(\psi(t)D^{2s-1}(\partial_{x}\varphi)u_{n}, u_{n}, u_{n} \right)| \leq C\|u_{n}\|_{X_{T}^{r, \frac{1}{2}}}\|v_{n}\|_{X_{T}^{r,-\frac{1}{2}}} \leq C.
\end{eqnarray*}
Analogously, 
\begin{eqnarray*}
|(\psi(t)D^{2s-1}(\partial_{x}\varphi)\partial_{x}u_{n}, u_{n})| \leq C.
\end{eqnarray*}
Since, $fu \in L^{2}_{loc}(0, T; H^{s}(\mathbb{T}))$ and 
$f \partial_{x}u \in L^{2}_{loc}(0, T; H^{s-1}(\mathbb{T}))$ we can conclude that $$fu_{n} = e^{\frac{1}{n}\partial_{x}^{2}}fu + \left[f, e^{\frac{1}{n}\partial_{x}^{2}}\right]$$ is uniformly bounded 
in $L^{2}_{loc}(0, T; H^{s}(\mathbb{T}))$, since 
$s \leq r+1$. The same reasoning we get that $f \partial_{x}u$ is uniformly bounded 
in $L^{2}_{loc}(0, T; H^{s-1}(\mathbb{T}))$ thus, 
\begin{eqnarray*}
|(\psi(t)D^{2s-1}f\partial_{x}u, D^{s}fu_{n})| \leq C.
\end{eqnarray*}
Due to the fact that $[D^{s}, f]$ is a pseudo-differential operator of order $s-1$ and $u \in L^{2}(0, T; H^{r}(\mathbb{T}))$, we have
\begin{equation*}
\begin{split}
|(\psi(t)D^{s-1}f\partial_{x}u_{n}, [D^{s}, f]u_{n})|  \leq & \|D^{s-1}fu_{n}\|_{L^{2}(0, T;L^{2}(\mathbb{T}))} \|D^{\rho}[D^{\rho}, f]u_{n}\|_{L^{2}(0, T; L^{2}(\mathbb{T}))}\\
  \leq & \|u_{n}\|_{L^{2}H^{r}}\|u_{n}\|_{L^{2}H^{s -1 + \rho}} \leq C
 \end{split}
\end{equation*}
and finally, 
$$
|(\psi(t)D^{2s-1}(\partial_{x}\varphi)\partial_{x}u_{n}, u_{n})| \leq C.
$$
As in proposition \ref{Prop-Comp}, we finish this proof by 
writing $\partial \varphi = f_{2}(\cdot) - f_{2}(\cdot - x_{0})$, with $x_{0} \in \mathbb{T}$, 
thus 
$$
|(\psi(t)D^{2s-1}f_{2}(\cdot - x_{0})\partial_{x}u_{n}, u_{n})| \leq C. 
$$
From this calculation, we conclude that 
$u \in L^{2}_{loc}(0, T; H^{r + \rho}(\mathbb{T}))$. 

To conclude the proof, we argue analogously for the KdV operator $L_{2} = \partial_{t} + \partial_{x}^{3} + \mu \partial_{x}$, and then we consider
\begin{equation*}
\widetilde{\alpha}_{n} = \left([A, L_{2}]v_{n}, v_{n} \right)_{L^{2}(0, T; L^{2}_0(\mathbb{T}))} \\
 =  (Lv_{n}, A^{*}v_{n})_{L^{2}_{t,x}} + (Av_{n}, L_{2}v_{n})_{L^{2}_{t,x}}
\end{equation*}
and we also have, 
\begin{eqnarray*}
\widetilde{\alpha}_{n} = (\operatorname{Re}(\partial_{x}u_{n}), A^{*}v_{n})_{L^{2}_{t,x}} + (Av_{n}, g_{n})_{L^{2}_{t,x}} + (g_{n}, A^{*}v_{n})_{L^{2}_{t,x}} + (Av_{n}, \operatorname{Re}(\partial_{x}u_{n}))_{L^{2}_{t,x}}   
\end{eqnarray*}
and the main difference from the first 
the equation is the estimate of 
\begin{eqnarray*}
(\operatorname{Re}(\partial_{x}u_{n}), A^{*}v_{n})_{L^{2}_{t,x}} \quad \quad \hbox{ and } \quad  \quad (Av_{n}, \operatorname{Re}(\partial_{x}u_{n}))_{L^{2}_{t,x}} , 
\end{eqnarray*}
which can be dealt with as in \eqref{prop-comp1}, and so the propagation of regularity is proved.
\end{proof}

\begin{remark}
We mention that the main ingredient in the proof of Propositions \ref{Prop-Comp} and \ref{prop-reg} is related to estimating the terms 
$$(i\partial_{x}v_{n}, A_{\epsilon}^{*}u_{n})_{L^{2}_{t,x}},$$
$$(Re(\partial_{x}u_{n}), v_{n}) +
(v_{n}, Re(\partial_{x}u_{n})),$$
and
$$(\operatorname{Re}(\partial_{x}u_{n}), A^{*}v_{n})_{L^{2}_{t,x}}+(Av_{n}, \operatorname{Re}(\partial_{x}u_{n}))_{L^{2}_{t,x}},$$ that is, where the coupled terms are analyzed.  In this sense Lemma \ref{LemaBou} is essential to prove the propagation results for the system under consideration in this work. 
\end{remark}

The next result is a consequence of Proposition \ref{prop-reg}.

\begin{corollary}\label{Cor-Reg}
Let $(u,v) \in X^{r, \frac{1}{2}}_{T} \times Y^{r, \frac{1}{2}}_{T}$ be a solution of
\begin{equation}
\left\{\begin{array}{ll}
i \partial_{t} u+\partial_{x}^{2} u=i \partial_{x} v+\beta|u|^{2} u,  & (x,t)\in \mathbb{T} \times (0,T),\\
\partial_{t} v+\partial_{x}^{3} v + \mu \partial_{x}v +\frac{1}{2} \partial_{x}\left(v^{2}\right)+\mu \partial_{x} v= \operatorname{Re}\left(\partial_{x} u\right) ,  & (x,t)\in \mathbb{T} \times (0,T).
\end{array}\right.
\end{equation}
Assume that, for some nonempty open set $\omega \subset \mathbb{T}$, $(u, v) \in \left[ C^{\infty}(\omega \times (0, T))\right]^{2} $. So $(u, v) \in [C^{\infty}(\mathbb{T} \times (0, T))]^{2}$.
\end{corollary}

\begin{proof}
Note that, we may assume that 
$[v] = 0$, thus we can apply the estimate \eqref{Est-Bi} to conclude that $\partial_{x}(v^{2}) \in Y_{T}^{r, -\frac{1}{2}}$ and we have that $\beta|u|^{2}u \in X^{r, -\frac{1}{2}}_{T}$ by estimate \eqref{Est-Tri}. Then we can apply the 
Proposition $\ref{prop-reg}$ to conclude that 
$(u, v) \in L^{2}_{loc}(0, T; H^{r + \rho}(\mathbb{T}) \times H^{r + \rho}(\mathbb{T}))$. Hence we can get 
$t_{0} \in (0, T)$ such that $(u(t_{0}),v(t_{0})) \in H^{r+\rho}(\mathbb{T}) \times H^{r + \rho}(\mathbb{T})$, then solving the equation with $(u(t_{0}), v(t_{0}))$ as initial data, it follows that 
$(u, v) \in L^{2}(0, T; H^{r+\rho}(\mathbb{T}) \times H^{r + \rho}(\mathbb{T}))$ and iterating this process we get that if $(u, v) \in \left[C^{\infty}(\mathbb{T}\times (0, T)) \right]^{2}$, showing the result.
\end{proof}

\section{Unique continuation for the NLS--KdV system}\label{Appendix2}

We present the unique continuation property for the NLS--KdV system. This will be proved by taking advantage of Carleman estimates already proved for the Korteweg-de Vries and Schrödinger operators in \cite{Arauna,BousinesqR}.  For self-contentedness, we will present below the estimates necessary to demonstrate the result.

\subsection{Carleman estimates}  Consider the KdV equation on $2\pi$-periodic conditions
\begin{equation}\label{A1}
\left\{\begin{array}{ll}
\partial_tq + \partial^3_xq =f, &  (x,t)\in (0,2\pi)\times(0,T) , \\
q(0,t)=q(2\pi,t)  & t\in(  0,T)  \text{,}\\
\partial_xq(0,t) =\partial_xq(2\pi,t) &  
t\in(  0,T)  \text{,}\\
\partial^2_xq(0,t)=\partial^2_xq(2\pi,t)  &
t\in(  0,T)  \text{,}\\
q(x,0) =q_0(x), &  x\in (0,2\pi).
\end{array}\right.
\end{equation}
Then the following Carleman inequality is a consequence of the result proved in \cite[Proposition 3.1]{BousinesqR}.
\begin{proposition}Pick any $T>0$. 
There exist two constants $C>0$ and $s_0 >0$ such that any $f\in L^2(0,T;L^2(0,2\pi))$, any $q_0\in L^2_0(0,2\pi)$ and any  $s\ge s_0$, the solution $q$ of
\eqref{A1} fulfills
\begin{equation}\label{D60}
\begin{split}
\int_0^T\!\! \int_0^{2\pi} &[s\vf |q_{xx}|^2 +(s\vf )^3 |q_x|^2 +(s\vf )^5 |q|^2 ]e^{-2s\vf} dx dt \le C\int_0^T\!\! \int_0^{2\pi}|f|^2e^{-2s\vf}dxdt\\
&+ C\left(  \int_0^T\!\! \int_{\omega}[s\vf |q_{xx}|^2 +(s\vf )^3 |q_x|^2 +(s\vf )^5 |q|^2 ]e^{-2s\vf} dx dt \right).
\end{split}
\end{equation}
\end{proposition}

Let us now borrow the results obtained in \cite[Theorem 3.2.]{Arauna}. Consider the Schrödinger equation on $2\pi$-periodic conditions
$$
\left\{\begin{array}{lll}
i\partial_t p+\partial^2_xp=g& (x,t)\in (0,2\pi)\times(0,T). \\
p(0, t)=p(2\pi, t)=0 & t\in (0, T), \\
\partial_xp(0, t)=\partial_xp(2\pi, t)=0 & t\in (0, T), \\
p(x,0)=p_0(x) & x\in(0,2\pi).
\end{array}\right.
$$
Thus, the following estimate holds.
\begin{theorem}
There exist constants $C>0$ and $s_{0} \geq 1$ such that
\begin{equation}\label{D61}
\begin{split}
\int_0^T\!\! \int_0^{2\pi} [s\vf \left|p_{x}\right|^{2}+(s\vf)^{3} |p|^{2} ]e^{-2s\vf}d x d t \leq& \int_0^T\!\! \int_0^{2\pi}|g|^2e^{-2s\vf}dxdt\\&+C  \int_0^T\!\! \int_{\omega} [(s\vf)^{3}|p|^{2} + (s\vf)\left|\operatorname{Re}\left(p_{x}\right)\right|^{2}] e^{-2s \vf} d x d t
\end{split}
\end{equation}
for all $s>s_{0}$ and $\omega\subset(0,2\pi)$.
\end{theorem}

\subsection{Unique continuation property} Carleman inequalities with an internal observation for the KdV and Schrödinger operators ensures us the result of the unique continuation.
\begin{theorem}\label{Carleman}
Let $T > 0$ and $V_{1}, V_{2}, V_{3} \in C^{\infty}((0, T) \times \mathbb{T}))$, then the solution of 
\begin{equation*}
\left\{\begin{array}{ll}
 i \partial_{t}u + \partial_{x}^{2}u = i \partial_{x}v + V_{1}(x,t)u + V_{2}(x,t)\overline{u},  & (x,t)\in \mathbb{T} \times (0,T),\\
\partial_{t}v + \partial_{x}^{3}v +\mu \partial_{x}v = V_{3}(t, x)\partial_{x}u + \operatorname{Re}(\partial_{x}u),  & (x,t)\in \mathbb{T} \times (0,T),\\
u(x,t) = v(x,t) = 0, & (x,t)\in \omega \times (0, T),
\end{array}\right.
\end{equation*}
with $(u,v) \in C([0, T]; L^{2}(\mathbb{T}) \times L^{2}_{0}(\mathbb{T}))$, is the trivial one $u \equiv v \equiv 0$ in $\mathbb{T} \times (0,T)$ .
\end{theorem} 
\begin{proof}
Putting together \eqref{D60} and \eqref{D61}, we have that 
\begin{equation}\label{D62}
\begin{split}
\int_0^T\!\! \int_0^{2\pi} &[s\tilde{\vf} \left|u_{x}\right|^{2}+(s\tilde{\vf})^{3} |u|^{2} ]e^{-2s\hat{\vf}}d x d t\\&+
\int_0^T\!\! \int_0^{2\pi} [s\tilde{\vf} |v_{xx}|^2 +(s\tilde{\vf} )^3 |v_x|^2 +(s\tilde{\vf})^5 |v|^2 ]e^{-2s\tilde{\vf}}dxdt\\\leq& \int_0^T\!\! \int_0^{2\pi}(|g|^2+|f|^2)e^{-2s \hat{\vf}}dxdt+C  \int_0^T\!\! \int_{\omega} [(s\vf )^{3}|u|^{2} + (s\vf)\left|\operatorname{Re}\left(u_{x}\right)\right|^{2}] e^{-2s\vf} d x d t
\\
&+ C\left(  \int_0^T\!\! \int_{\omega}[s\vf |v_{xx}|^2 +(s\vf )^3 |v_x|^2 +(s\vf )^5 |v|^2 ]e^{-2s\vf} dx dt \right),
\end{split}
\end{equation}
where $\tilde{\vf}=\min{\vf}$, $\hat{\vf}=\max{\vf}$, $g:= i \partial_{x}v + V_{1}(x,t)u + V_{2}(x,t)\overline{u}$ and $f:= V_{3}(t, x)\partial_{x}u + \operatorname{Re}(\partial_{x}u)$.  The properties of the $\tilde{\vf}$ and $\hat{\vf}$ can be seen in \cite{Arauna,BousinesqR}.

Now, note that the zero and first-order terms of the right-hand side of \eqref{D62} can be absorbed by the left-hand side. Moreover, the second order term $v_{xx}$ of the right-hand side \eqref{D62} can be treated as in \cite[Lemma 3.7]{BousinesqR}, resulting in the following estimate  \begin{equation}\label{D63} \int_0^T\!\! \int_0^{2\pi} [(s\tilde{\vf})^{3} |u|^{2}+(s\tilde{\vf})^5 |v|^2 ]e^{-2s\tilde{\vf}}dxdt \leq C  \int_0^T\!\! \int_{\omega} [(s\vf )^{3}|u|^{2} +(s\vf )^5 |v|^2 ]e^{-2s\vf} dx dt. \end{equation} Therefore, thanks to the hypothesis $u \equiv v \equiv 0$  in  $\omega \times (0, T)$, the right hand side of \eqref{D63} yields that $u \equiv v \equiv 0$ in $\mathbb{T} \times (0,T)$, and so the result follows.
\end{proof}

As a consequence of this unique continuation property, we have the following result.

\begin{corollary}\label{Uniq-Cont-Cor}
Let $(u, v) \in X^{0, \frac{1}{2}}_{T} \times Y^{0, \frac{1}{2}}_{T}$ be a solution of 
\begin{equation}\label{Uniq-Cont}
\left\{\begin{array}{ll}
i \partial_{t} u+\partial_{x}^{2} u=i \partial_{x} v+\beta|u|^{2} u, & (x,t)\in \mathbb{T} \times (0,T),\\
\partial_{t} v+\partial_{x}^{3} v+\frac{1}{2} \partial_{x}\left(v^{2}\right)+\mu \partial_{x} v=-R e\left(\partial_{x} u\right),  & (x,t)\in \mathbb{T} \times (0,T),\\
u(x,t)=0, \quad v(x,t)=c, &(x,t)\in \omega \times(0, T),
\end{array}\right.
\end{equation}
where $\omega \subset \mathbb{T}$ is a nonempty open set and $c$ is some real constant. Then, there exists a small $T > 0$, such that $u = 0$ on $\mathbb{T} \times (0, T)$ and $v = c$ on $\mathbb{T} \times (0, T)$.
\end{corollary}
\begin{proof}
First note that using Corollary \ref{Cor-Reg} we have that $u, v \in C^{\infty}(\mathbb{T} \times (0, T))$ then, considering $(w,z) = (\partial_{t}u, \partial_{t}v)$ we get 
\begin{equation*}
\left\{\begin{array}{ll}
 i \partial_{t}w + \partial_{x}^{2}w = i \partial_{x}z + V_{1}(t,x)w + V_{2}(t, x)\overline{w},& (x,t)\in \mathbb{T} \times (0,T),\\
 \partial_{t}z + \partial_{x}^{3}z + \mu \partial_{x}z + V_{3}(t, x)z = Re(\partial_{x}w),& (x,t)\in \mathbb{T} \times (0,T),\\
 z(x,t)= w(x,t) = 0,&(x,t)\in \omega \times (0, T).
\end{array}\right.
\end{equation*}
Then the unique continuation property given by Theorem \ref{Carleman} implies that $w = z = 0$. In analogous way, we can consider $(\widetilde{w}, \widetilde{z}) = (\partial_{x}u, \partial_{x}v)$ and then we get that $\partial_{x}u = \partial_{x}v = 0$ in $\mathbb{T} \times (0, T)$. Hence we can conclude that both functions $u, v$ are constant and the result follows.
\end{proof}

 \subsection*{Acknowledgments} We greatly appreciate the referees’ careful reading and helpful suggestions. Capistrano--Filho was supported by CNPq 307808/2021-1, CAPES-PRINT 88881.311964/2018-01, CAPES-MATHAMSUD 88881.520205/2020-01, MATHAMSUD 21- MATH-03 and Propesqi (UFPE).  Pampu was supported by PNPD-CAPES grant number 88887.351630/2019-00. This work was done during the postdoctoral visit of the second author at the Universidade Federal de Pernambuco, who thanks the host institution for the warm hospitality.

\end{document}